\documentclass[11pt,reqno]{amsart}
\usepackage{amsmath,amssymb,amsthm,mathtools,wasysym,calc,verbatim,enumitem,tikz,url,hyperref,mathrsfs,cite,fullpage,bbm,yhmath}

\usepackage{comment}

\usepackage{cleveref}

\newtheorem{theorem}{Theorem}

\newtheorem{lemma}[theorem]{Lemma}
\newtheorem{claim}[theorem]{Claim}
\newtheorem{corollary}[theorem]{Corollary}

\newtheorem*{claim*}{Claim}

\theoremstyle{remark}
\newtheorem*{remark*}{Remark}

\numberwithin{theorem}{section}

\renewcommand{\phi}{\varphi}

\renewcommand{\leq}{\le}
\renewcommand{\geq}{\ge}

\newcommand{\eps}{\varepsilon}

\newcommand{\cF}{\mathcal F}
\newcommand{\cG}{\mathcal G}
\newcommand{\cE}{\mathcal E}

\newcommand{\PP}{\mathbb P}

\newcommand{\cR}{\mathcal R}

\newcommand{\E}{\mathbb{E}}

\def\1{\mathbbm{1}}

\newcommand{\la}{\langle}
\newcommand{\ra}{\rangle}

\def\g{{\gamma}}
\def\s{{\sigma}}

\renewcommand{\le}{\leqslant}
\renewcommand{\ge}{\geqslant}
\renewcommand{\P}{\mathbb{P}}

\newcommand{\R}{\mathbb R}

\newcommand{\bS}{\mathbb S}

\newcommand{\snorm}[1]{\lVert#1\rVert}
\newcommand{\bnorm}[1]{\bigg\lVert#1\bigg\rVert}
\newcommand{\sang}[1]{\langle #1 \rangle}

\allowdisplaybreaks

\newcommand{\mb}{\mathbb}

\newcommand{\mbm}{\mathbbm}
\newcommand{\mc}{\mathcal}

\newcommand{\mr}{\mathrm}

\newcommand{\on}{\operatorname}

\newcommand{\wh}{\widehat}
\newcommand{\wt}{\widetilde}

\title{On the Spielman-Teng Conjecture} 

\author[A1]{Ashwin Sah}
\address{Department of Mathematics, Massachusetts Institute of Technology, Cambridge, MA 02139, USA}
\email{asah@mit.edu}

\author[A2]{Julian Sahasrabudhe}
\address{University of Cambridge. Department of Pure Mathematics and Mathematical Statistics.}
\email{jdrs2@cam.ac.uk}

\author[A3]{Mehtaab Sawhney}
\address{Department of Mathematics, Columbia University, New York, NY 10027}
\email{m.sawhney@columbia.edu}

\begin{document}
	\begin{abstract}
Let $M$ be an $n\times n$ matrix with iid subgaussian entries with mean $0$ and variance $1$ and let $\sigma_n(M)$ denote the least singular value of $M$. We prove that 
\[\PP\big( \sigma_{n}(M) \leq \eps n^{-1/2} \big) = (1+o(1)) \eps + e^{-\Omega(n)}\]
for all $0 \leq \eps \ll 1$. This resolves, up to a $1+o(1)$ factor, a seminal conjecture of Spielman and Teng. 
	\end{abstract}	
\maketitle

\vspace{-2em}

\section{Introduction}
For an $n\times n $ matrix $A$, the \emph{least singular value} of $A$ is defined to be $\sigma_{n}(A) = \min_{x \in \mb{S}^{n-1}} \|Ax\|_2$. This fundamental quantity has been extensively studied in the context of random matrices and goes back, at least, to the seminal work of von Neumann \cite{von63}, in the 1960s, on approximate solutions to linear systems of equations. More recently, it has played a central role in breakthroughs on the limiting spectral laws of random matrices, for example in the proof of the famous “circular law” of Tao and Vu \cite{TV10d}, and on the “smoothed analysis" of algorithms \cite{SST06}.

In this paper we consider large matrices with iid entries; let $M$ be an $n\times n$ iid random matrix with entries that have mean $0$ and variance $1$. Here we expect that $\sigma_{n}(M) \approx n^{-1/2}$ and one is naturally led, in theory and applications, to the study of the quantity
\begin{equation}\label{eq:LSV-prob} \PP\big( \sigma_{n}(M) \leq  \eps n^{-1/2} \big).\end{equation} 

Aside from the several prominent applications mentioned above, this research direction has been stimulated by the work of Spielman and Teng, who formulated a precise conjecture concerning \eqref{eq:LSV-prob} in the case of \emph{Rademacher} random matrices, that is, when the entries of $M$ are iid uniform in $\{-1,1\}$. Let $B$ be such an $n\times n$ Rademacher matrix; they conjectured that
\begin{equation}\label{eq:ST-conj} \PP\big( \sigma_{n}(B)\le \eps n^{-1/2} \big) \leq \eps + e^{-\Omega(n)}, \end{equation}
for all $\eps\geq 0$.  This conjecture, put forward in their 2002 ICM survey \cite{ST02}, has since become known as the \emph{Spielman--Teng conjecture} and has stimulated a great deal of work on the least singular value of iid matrices in the past 20 years. Most notably Rudelson and Vershynin \cite{RV08} proved \eqref{eq:ST-conj} up to a constant factor, and Tao and Vu \cite{TV10} proved \eqref{eq:ST-conj} in the regime when $\eps \geq n^{-c}$.
 
In this paper we prove the Spielman--Teng conjecture, up to a $1+o(1)$ factor, for all $\eps \geq 0$. In fact, we prove the same result holds for all random matrices with iid subgaussian entries. 

\begin{theorem}\label{cor:ST} Let $M$ be an $n\times n$ random matrix with iid subgaussian entries $\xi$ with mean $0$ and variance $1$. Then, for all $\eps \geq 0$,
\begin{equation} \label{eq:ST} \PP\big( \sigma_{n}(M)\le \eps n^{-1/2} \big) \leq
(1 + o(1) ) \eps + e^{-\Omega_{\xi}(n)}, \end{equation}
where the $o(1)$ term decays as $C_{\xi}(\log n)^{-1/16}$, where $C_{\xi} >0$ depends only on $\xi$.
\end{theorem}

We suspect that our theorem is best possible in the sense that the $1+o(1)$ factor cannot be removed for this wider class of matrices. Indeed, for \emph{complex valued} iid matrices, examples were given by Edelman, Guionnet, and P\'{e}ch\'{e} \cite{EGP16}, suggesting that the $1+o(1)$ factor is indeed necessary. In other words, we believe that the specific bound \eqref{eq:ST-conj} conjectured by Spielman and Teng, if true, is not a ``universal phenomenon'', independent of the entry distribution, but depends specifically on the properties of the uniform distribution on $\{-1,1\}$.

We also show the corresponding \emph{lower bound} in Theorem~\ref{cor:ST} when $\eps = o(1)$. We do this by showing that the distribution of $\sigma_n(M)$ agrees with the corresponding quantity for a Gaussian random matrix, all the way down to exponentially small scales, up to a $1+o(1)$ factor.

\begin{theorem}\label{thm:main}
Let $M$ an $n\times n$ random matrix with iid subgaussian entries $\xi$ of mean $0$ and variance $1$ and let $G$ be an $n\times n$ matrix with iid standard normal entries. Then, for all $\eps \geq 0$,
\begin{equation}\label{eq:thm-main}\PP\big( \sigma_{n}(M)\le \eps n^{-1/2} \big) = 
\big(1 + o(1) \big)\PP\big( \s_{n}(G)\le \eps n^{-1/2} \big) + e^{-\Omega_{\xi}(n)},\end{equation} where the $o(1)$ term decays as  $C_{\xi}(\log n)^{-1/16}$, where $C_{\xi}>0$ depends only on $\xi$.
\end{theorem}

Since the distribution of the least singular value is well understood in the Gaussian case, due to the work of Edelman \cite{Ede88} in the 1980s, we can quickly deduce the following. 

\begin{corollary}\label{cor:sharp-LSV-asymptotic}
Let $M$ an $n\times n$ random matrix with iid subgaussian entries $\xi$ of mean $0$ and variance $1$. Then 
\[ \PP\big( \sigma_{n}(M)\le \eps n^{-1/2} \big)= (1+o(1))\eps + e^{-\Omega_{\xi}(n)}, \]
for all $0 \leq \eps \ll 1$.
\end{corollary}

\subsection{Universality in random matrix theory} These results can be cast in the context of the wider project of proving ``universality'' for eigenvalue statistics of random matrices. This idea goes back to the pioneering work of Wigner in the 1950s, who put forward the bold hypothesis that the behaviour of the eigenvalues of random matrices depends very little on the underlying distribution of the entries. 

In recent years there have been spectacular advances towards the project of proving universality of various random matrix statistics. Prominent among these is the famous  ``circular law'', which was finally settled in full by Tao and Vu \cite{TV10d} after a long sequence of advances \cite{Meh67,Gir84,Ede97,Bai97,GT10}. This result tells us that the macroscopic ``shape'' of the eigenvalues $M$ (i.e.~the spectral distribution) is largely independent of the entry distribution and converges in the limit to the uniform distribution on the unit disc, when properly renormalized.

Another flagship advance has come in the context of random \emph{symmetric} matrices: let $A$ be an $n\times n$ random symmetric matrix where the entries on and above the diagonal are iid with mean $0$ and variance $1$. Also let $\Lambda(A) = \{ \lambda_1(A) n^{-1/2},  \ldots, \lambda_n(A)n^{-1/2} \}$ be the set of (renormalized) eigenvalues. The extraordinary results of Tao and Vu \cite{TV11} and Erd\H{o}s, Schlein and Yau \cite{ESY11} tell us that if $I \subseteq (-2,2)$ is an interval with $|I| = \Theta(n^{-1})$ then
\begin{equation}\label{eq:Universality} 
\PP\big( \Lambda(A) \cap I\neq \emptyset \big)   = \PP\big( \Lambda(G_0) \cap I \neq \emptyset\big) + \Theta(n^{-c}), \end{equation}
where $G_0$ is the $n\times n$ random symmetric matrix with iid standard normal entries on and above the diagonal. Informally \eqref{eq:Universality} states that the probability that a (polynomially) small interval contains a eigenvalue is independent of the underlying entry distribution. More generally, the results of \cite{TV11,ESY11} imply that the microscopic distribution of the eigenvalue gaps is independent of the underlying distribution of the entries, up to probability scales of order $n^{-c}$. We note that a polynomial probability error is expected in these statements; in particular the dependence on $\E[A_{ij}^4]$ appears when considering fine asymptotics of the mean of $\lambda_i(A)$ (see \cite[Theorem~1.4]{LS20}).

Another important universality result, and most relevant for us here, was proved by Tao and Vu \cite{TV10}, in the setting of iid matrices. They showed that the distribution of the least singular value of $\sigma_n(M)$ is similarly universal. Specifically they proved that 
\begin{equation}\label{eq:TV} \PP\big( \sigma_n(M) \leq \eps n^{-1/2} \big) = \PP\big( \sigma_n(G) \leq \eps n^{-1/2} \big) + \Theta(n^{-c}),  \end{equation}  
where $G$ is an iid random matrix with standard normal entries. Again, since the distribution of $\sigma_n(G)$ is known exactly (due to the special symmetries in this case), Tao and Vu were able to deduce that the Spielman--Teng conjecture holds for all $\eps \geq n^{-c}$. 

In this paper we make what is perhaps the first incursion into the project of studying a ``sub--microscopic'' universality phenomena, by showing that this same universality phenomena persists all the way down to \emph{exponentially small scales}, if one allows for a $1+o(1)$ factor. 

\subsection{History of the least singular value problem} The study of the least singular value of iid random matrices goes back to the work of von Neumann~(see \cite[pg.~14,~477,~555]{von63}) in the context of computing approximate solutions to systems of linear equations. He suggested that if $B$ is an $n\times n$ Rademacher random matrix then one has 
\begin{equation}\label{eq:smale} \s_{n}(B) \approx n^{-1/2}.\end{equation} 
This was later formally conjectured by Smale \cite{Sma85} and then proved by Szarek \cite{Sza91} and Edelman \cite{Ede88} when $B$ is replaced with the $n\times n$ matrix $G$ with iid standard Gaussian entries. In this case, Edelman found an \emph{exact} expression for the density of the least singular value which implies that
\begin{equation}\label{eq:edelman}
 \PP\big( \s_{n}(G)  \leq \eps n^{-1/2} \big) \leq \eps, 
\end{equation}
for all $\eps \geq 0$ (see e.g.~\cite{ST02}). While this gives us an essentially complete picture in the Gaussian case, Edelman's proof relies fundamentally on the explicit formulae for the distribution of the singular values of $G$ that are not available for other entry distributions. 

One challenging feature of general entry distributions, and in particular of \emph{discrete} distributions, is that such matrices are singular with non-zero probability. This singularity event corresponds to the case $\eps = 0$ and exhibits a very different behaviour from the situation for $\eps \geq e^{-o(n)} $. In fact, the problem of estimating the singularity probability of discrete random matrices has enjoyed something of its own history, going back to the pioneering work of Koml\'{o}s \cite{Kom67} in the 1960s, and has since been the subject of intense activity \cite{KKS95,TV07,RV08,BVW10}. Today, the best known result on random Radamacher matrices is due to Tikhomirov who, in a breakthrough paper \cite{Tik20}, proved   
\[ \PP( \det(B) = 0 ) = 2^{-n + o(n) }. \]

For general $\eps \geq 0$, a key breakthrough was made by Rudelson \cite{Rud08}, and independently by Tao and Vu \cite{TV09}, who gave the first bounds on the least singular value beyond the Gaussian case. Then Rudelson and Vershynin \cite{RV08} proved 
\begin{equation}\label{eq:RVapproxST}
 \PP\big( \s_{n}(B)  \leq \eps n^{-1/2} \big) \leq C\eps + e^{-cn},
\end{equation}
thereby resolving the Spielman--Teng conjecture, up to constant factors and, as a consequence, Smale's original problem \eqref{eq:smale}.

This was then later complemented by the work of Tao and Vu \cite{TV10}, mentioned above \eqref{eq:TV}, who proved that the constant $C$ can be taken to be $1$ in the case $\eps \geq n^{-c}$. In this work we, in a sense, ``unite'' these two results by showing the constant $C$ is universal and equal to $1$ for all $\eps \geq 0$.

In the next section we give an outline of the proof Theorem~\ref{thm:main}, which implies Theorem~\ref{cor:ST} and Corollary~\ref{cor:sharp-LSV-asymptotic}. We then proceed to prove Theorem~\ref{thm:main} in Section~\ref{sec:geometric-reduction} through Section~\ref{sec:return-to-sing-value}.

\section{Outline of proof}\label{sec:outline}
Throughout we assume that $M$ is an $n\times n$ matrix with iid subgaussian entries. We assume that  $\eps \geq e^{-cn}$, and that $\eps < n^{-c}$ for some suitablly small $c>0$. Otherwise we can apply the theorem of Rudelson and Vershynin, or the theorem of Tao and Vu, respectively. To reduce on clutter, in our discussion here, we will not worry about tracking the exact error term and be content with multiplicative losses of $1+o(1)$. We shall also just focus our discussion here on proving the ``$\leq$'' direction of Theorem~\ref{thm:main}, since this is the direction that implies Theorem~\ref{cor:ST}. Of course, we have corresponding lower bounds for each of the crucial steps below. 

\subsection{A geometric reduction} The first key idea in the work of Rudelson and Vershynin is to show that one can control the event $\s_n(M) \leq \eps n^{-1/2}$ in \emph{geometric} terms. In particular, the first key step in their proof is to show
\begin{equation}\label{eq:RV-distance-prob} \PP\big( \sigma_{n}(M) \leq \eps n^{-1/2} \big) \leq C \cdot \E_{M^{\ast}} \,  \P_X\big( |\la v, X\ra| \leq \eps \big)  + o(\eps), \end{equation}
where $C>1$ is an absolute constant, $X$ denotes the last row of the matrix $M$, $M^{\ast}$ denotes the matrix $M$ with the last row removed, and $v$ is an arbitrary unit vector in $\ker(M^{\ast})$. This simple and powerful idea
allows one to access the least singular value $\sigma_n$ through an understanding of the kernel vector $v$ and its inner product with an independent random vector $X$. However, for us, there is a serious drawback in this first step as we necessarily lose a crucial constant factor $C>1$. 

In this paper, our first step draws inspiration from \eqref{eq:RV-distance-prob}, though is much more refined and (as one might imagine) more difficult to work with. Our version of \eqref{eq:RV-distance-prob} takes the shape
\begin{equation}\label{eq:RV-loss-less}
 \PP\big( \sigma_{n}(M) \leq \eps n^{-1/2} \big) \leq  
\E_{M^{\ast}}\, \PP_X\left( |\sang{v,X}| \leq \big( 1+\eps^{1/4}\big)\eps n^{-1/2}    \wt{\chi}(X) \right) + o(\eps),
\end{equation}
where $v \in \ker(M^{\ast})$ is a unit vector and $\wt{\chi}(X) = \wt{\chi}_{M^{\ast}}(X)$ is the ``correction'' term defined by 
\begin{equation}\label{eq:correction-sum} \wt{\chi}^2(X) =  1+ \sum_{i=1}^{n-1} \frac{\sang{v_i,X}^2}{\s_{i}(M^{\ast})^2} .\end{equation} Here $v_{n-1},\ldots, v_1$ are the unit singular vectors corresponding to $\s_{n-1}(M^{\ast}) \leq \cdots \leq \s_1(M^{\ast})$, the singular values of the matrix $M^{\ast}$. Note here that once we fix $M^{\ast}$, the probability on the right hand side of \eqref{eq:RV-loss-less} depends only on the inner product of $X$ with singular directions of $M^{\ast}$, in the geometric spirit of \eqref{eq:RV-distance-prob}.

\subsection{A truncation step} The major challenge in working with \eqref{eq:RV-loss-less} is, perhaps unsurprisingly, getting a handle on the sum \eqref{eq:correction-sum}, which fluctuates at the scale of $\Theta(n^{1/2})$. Here our first major step in the paper is to show that we can truncate this sum to the terms corresponding to the smallest $\sqrt{\log n}$ singular values, while conceding only marginal losses. More precisely, in Sections~\ref{sec:truncation-initial} through \ref{sec:truncation-completion}, we prove the following ``truncated'' version of \eqref{eq:RV-loss-less}
\begin{equation}\label{eq:main-distance-to-subspace}
 \PP\big( \sigma_{n}(M) \leq \eps n^{-1/2} \big) \leq  
\E_{M^{\ast}} \, \PP_X\Big( |\sang{v,X}| \leq (1+\delta_n ) \eps n^{-1/2} \chi(X) \Big) + o(\eps),
\end{equation} where we define, here and throughout the paper, 
\begin{equation}\label{eq:def-delta-ell-chi} \delta_n = (\log n)^{-c},  \qquad  \quad  \ell = \sqrt{\log n}, \qquad \text{ and } \qquad  \chi^2(X) = \sum_{i=1}^{\ell}\frac{\sang{v_{n-i},X}^2}{\s_{n-i}(M^{\ast})^2} , \end{equation} for some $c>0$. 

The proof of \eqref{eq:main-distance-to-subspace} is fairly involved and will consume our attention from Section~\ref{sec:truncation-initial} to Section~\ref{sec:truncation-completion}. Our first step towards \eqref{eq:main-distance-to-subspace} is to show that \eqref{eq:main-distance-to-subspace} follows fairly easily if the matrix $M$ satisfies a ``regularity'' event $\cR$.
While it is a bit technical to define this event (as we do in Section~\ref{sec:truncation-initial}), for now it is enough to think of $\cR$ as an approximation of the intersection of the events 
\[ \sigma_{n-r}(M^{\ast}) \approx \E\, \sigma_{n-r}(M^{\ast}) \approx r n^{-1/2} \qquad \text{ and } \qquad  \la v_i, X \ra^2 \approx \E \, \la v_i, X \ra^2 = 1, \]
for each $r$ and $i$. Most of the work in these sections is in showing that the probability that $\cR$ fails is negligible. More precisely we show that 
\[ \PP\big( \sigma_n(M) \leq \eps n^{-1/2} \wedge \cR^c \big) = o(\eps). \]
To prove this we shall need to bootstrap several tools, developed in previous work, to deal with this complicated intersection of events. 
Of particular interest is the challenge in ``decoupling'' the intersections of events
\[  \sigma_n(M) \leq \eps n^{-1/2}\,  \wedge \, \la X, v\ra^2 \ll 1  \qquad \text{ and }  \qquad \sigma_n(M)\leq \eps n^{-1/2}\, \wedge \, \la X, v\ra^2 \gg 1,\] for which we use and develop the ``negative correlation inequalities'' introduced recently in the work of Campos, Jenssen, Michelen, and the second author in \cite{CJMS24} and \cite{CJMS22}. We also crucially lean on work of Rudelson and Vershynin \cite{RV09} on the least singular value of random rectangular matrices to control the event that $\sigma_{n-1}(M^{\ast})$ is atypically small.

\subsection{A Gaussian replacement step} Now to understand the right hand side of \eqref{eq:main-distance-to-subspace}, we think of fixing $M^{\ast} \in \cE^{\ast}$, where $\cE^{\ast}$ is the event that $M^{\ast}$ is appropriately quasi-random. We then consider the probability $\PP_X$ on the right hand side of \eqref{eq:main-distance-to-subspace} in isolation and relate this quantity to a similar quantity where the random vector $X$ is replaced with a \emph{Gaussian} random vector $Z = (Z_1,\ldots,Z_n)$, with iid standard Gaussian entries. In particular, we show that 
\begin{equation} \label{eq:G-replacement} \PP_X\Big( |\sang{v,X}| 
\leq \big( 1+\delta_n \big) \eps n^{-1/2} \chi(X) \Big) \leq \PP_Z\Big( |\sang{v,Z}| \leq \big( 1+2\delta_n \big) \eps n^{-1/2} \chi(Z) \Big) + o(\eps).  \end{equation}
This ``replacement'' maneuver echoes the so-called ``Lindeberg exchange'' method, which has been used to great effect in both random matrix theory (e.g.~in \cite{TV11}) and more widely in mathematics and computer science. The novelty in our replacement step \eqref{eq:G-replacement} is that we show that this exchange of $X$ for $Z$ can be done at probability scales of order $o(\eps)$, which, indeed, can be \emph{exponentially} small. This is in contrast to the above applications, where one obtains polynomial-type losses in  probability. To perform this ``exchange'', we rely crucially on the quasi-randomness properties of $M^{\ast} \in \cE^{\ast}$.

Once we have \eqref{eq:G-replacement}, we see that some of the challenge in working with the sum $\chi(X)$ immediately falls away. By the rotational invariance of $Z$ and the fact that the $v_i$ are orthonormal, we have that 
\begin{equation}\label{eq:distribution-chi} \left( \la Z, v \ra\, ,\, \sum_{i}\frac{\sang{v_{i},Z}^2}{\s_{i}(M^{\ast})^2} \right) \qquad \text{is distributed as } \qquad  \left( W_n \, ,\, \sum_{i} \frac{W_i^2}{\sigma_i(M^{\ast})^2} \right),\end{equation}
where the $W_i$ are iid standard normals.
 
\subsection{Rescaling and passing back to the singular value} With the observation \eqref{eq:distribution-chi} in tow, it is not too difficult to see that in right hand side of \eqref{eq:G-replacement}, we are able to ``replace'' $\eps$, which can be exponentially small, with a much larger (and therefore much more tractable) $\eps_0 = n^{-c'}$, for some appropriately small $c'>0$. More precisely, we show that 
\begin{equation}\label{eq:scaling-up} \PP_Z\Big( |\sang{v,Z}| \leq \big( 1+ 2\delta_n \big) \eps n^{-1/2} \chi(Z) \Big)  \leq ( \eps/\eps_0 ) \PP_Z\Big( |\sang{v,Z}| \leq \big( 1+ 2\delta_n \big) \eps_0 n^{-1/2} \chi(Z) \Big) + o(\eps). \end{equation}

Putting these steps together shows that
\begin{equation} \label{eq:after-X-replacement} \PP\big( \sigma_n(M) \leq \eps n^{-1/2} \big) \leq (\eps/\eps_0)\E_{M^{\ast}} \,  \PP_Z\Big( |\sang{v,Z}| \leq \big( 1+ \delta_n \big) \eps_0 n^{-1/2} \chi(Z) \Big) + o(\eps). \end{equation} We now turn to ``mix in'' the randomness of $M^{\ast}$.
The most direct route here would be to try to compute the expectation in \eqref{eq:after-X-replacement} by estimating the joint distribution of $\sigma_{n-1}(M^{\ast}),\ldots ,\sigma_{n-\ell}(M^{\ast})$. This is in fact possible, thanks to the work\footnote{Specifically, \cite[Theorem 6.2]{TV10}.} of Tao and Vu, which gives us a very good understanding of the spectrum when $\eps_0 \geq n^{-c}$, for some $c>0$. The trouble is that these results are technical to apply directly, so we instead take a simpler and more indirect route, by relating the quantity on the right hand side of \eqref{eq:after-X-replacement} \emph{back} to the probability that the least singular value of a certain matrix is small.

More precisely, we prove,  using a version of the inequality \eqref{eq:main-distance-to-subspace} \emph{in reverse}, that 
\begin{equation}\label{eq:back-to-sing-value-0} \E_{M^{\ast}} \,  \PP_Z\Big( |\sang{v,G}| \leq \big( 1+ \delta_n \big) \eps_0 n^{-1/2} \chi(Z) \Big)\leq \PP\big( \sigma_n(\wt{M}) \leq \eps_0 n^{-1/2}\big) + o(\eps_0), \end{equation}
where $\wt{M}$ is the matrix $M^{\ast}$ with a \emph{Gaussian} row $Z$ appended. We may now apply the powerful results of Tao and Vu which tell us that the distribution of $\sigma_n$ is \emph{universal} at these scales. Thus we can replace $\wt{M}$ with an iid Gaussian matrix $G$, for which the distribution of $\s_n$ is known. We have 
\begin{equation} \PP\big( \sigma_n(\wt{M}) \leq \big( 1+ \delta_n \big)\eps_0 n^{-1/2}\big) =\PP\big( \sigma_n(G) \leq \big( 1+ \delta_n \big)\eps_0 n^{-1/2}\big)  + O(n^{-c}) = \eps_0 + o(\eps_0). \end{equation}
This, when taken together with the above, proves our main theorem, Theorem~\ref{thm:main}.

\subsection{Organization of the paper} 
In Section~\ref{sec:geometric-reduction} we ``warm up'' by proving the first geometric reduction described at \eqref{eq:RV-loss-less}. We then endeavour to truncate the sum $\wt{\chi}$, which appears there in \eqref{eq:RV-loss-less} and prove \eqref{eq:main-distance-to-subspace}. This ``truncation step'' takes place over the course of Sections~\ref{sec:truncation-initial} to Section~\ref{sec:truncation-completion}.

In Section~\ref{sec:replacement}, we turn to prove our Gaussian replacement step, described at \eqref{eq:G-replacement}. Then, in Section~\ref{sec:return-to-sing-value} we prove the ``rescaling step'', 
described in \eqref{eq:scaling-up} and then pass back to the singular value problem in the matrix $\wt{M}$, a step described in \eqref{eq:back-to-sing-value-0}. We finally put the pieces together and prove Theorem~\ref{thm:main} and derive Theorem~\ref{cor:ST}.

\subsection{Global notation}\label{subsec:global}
In  an effort to reduce clutter, this paper is written with a few global assumptions on notation. Throughout $\xi$ with be a subgaussian random variable with subgaussian constant 
\[ B = \|\xi\|_{\psi_2} := \sup_{p \geq 1} p^{-1/2} \big( \E_{\xi}  \|\xi\|^p \big)^{1/p}. \]
Throughout $M$ will denote our $n\times n$ random matrix with entries distributed as $\xi$, where $n$ is fixed but sufficiently large. We also define $M^{\ast}$ to be the $(n-1)\times n$ matrix formed of the first $n-1$ rows of $M$ and we will let $X$ denote the last row of $M$.

Throughout we also assume that $v$ is a kernel vector of $M^{\ast}$ and that $v_{n-1},\ldots,v_1$ are the right singular vectors of  $M^{\ast}$ corresponding to $\sigma_{n-1}(M^{\ast}) \leq \cdots \leq \sigma_1(M^{\ast})$. As mentioned above at \eqref{eq:def-delta-ell-chi} we shall also assume \[ \delta_n = (\log n)^{-c},  \qquad  \quad  \ell = \sqrt{\log n}, \qquad \text{ and } \qquad  \chi^2(X) = \sum_{i=1}^{\ell}\frac{\sang{v_{n-i},X}^2}{\s_{n-i}(M^{\ast})^2} , \] for some $c>0$.

\section{An easy geometric reduction of the singular value problem}\label{sec:geometric-reduction}

In this short warm-up section we prove our easy initial geometric reduction of the singular value problem as described at \eqref{eq:RV-loss-less}. We will then go on to prepare for the rather involved proof of \eqref{eq:main-distance-to-subspace}.

\begin{lemma}\label{lem:geometric-reduction} Let $\eps \geq e^{-cn}$. Then 
\[\PP\big( \sigma_{n}(M) \leq \eps n^{-1/2} \big) \leq  
\PP\big( |\sang{v,X}| \leq \big( 1+\eps^{1/4} \big) \eps n^{-1/2}  \wt{\chi}(X) \big) + O(\delta_n\eps) \]
and conversely 
\[ \PP\big( \sigma_{n}(M) \leq \eps n^{-1/2} \big) \geq  
\PP\big( |\sang{v,X}| \leq \big( 1-\eps^{1/4} \big) \eps n^{-1/2}   \wt{\chi}(X) \big) - O(\delta_n\eps). \]  Here $c>0$ is an absolute constant.
\end{lemma}

\vspace{2mm}

Here $\wt{\chi}(X)$ is defined by $\wt{\chi}^2(X) =  1+ \sum_{i=1}^{n-1} \frac{\sang{v_i,X}^2}{\s_{i}(M^{\ast})^2}$, as we defined at \eqref{eq:correction-sum}. To prove this, we use the following deterministic ``rank-1 update'' formula from linear algebra (see e.g.~\cite[(5.1)]{Gol73}). We also include a short proof.

\begin{lemma}\label{lem:update}
Let $A$ be a  real $n\times n$ matrix, let $A^{\ast}$ be the matrix $A$ with the last row $Y$ removed, and let $u_n,\ldots, u_1$ be the orthogonal vectors corresponding to the singular values $0 = \s_n(A^{\ast} )\le \s_{n-1}(A^{\ast} )\le \cdots \le s_1(A^{\ast})$. Then the singular values of $A$ are the positive solutions, in $x$, to the polynomial equation
\[\prod_{i=1}^n(\sigma_i(A^\ast)^2-x)\cdot\bigg(1+ \sum_{i=1}^{n}\frac{\sang{u_i,Y}^2}{\s_i(A^{\ast})^2-x^2}\bigg) = 0.\]
\end{lemma}
\begin{proof}
Note that  $A^TA = (A^{\ast})^TA^{\ast} + Y^TY$. Thus, by the matrix determinant lemma, we have 
\[ \det(A^TA- x I)=\det((A^{\ast})^TA + Y^TY-x I)
=\det((A^{\ast})^TA -x I) \cdot (1+Y^{T}((A^{\ast})^TA -x I)^{-1}Y). \]
We then evaluate the right hand side to be 
\[ \prod_{i=1}^{n}(\s_i(A^{\ast})^2-x)\cdot \bigg(1+\sum_{i=1}^{n}\sang{u_i,Y}^2(\s_i(A^{\ast})^2-x)^{-1}\bigg),\]  
which implies the lemma. \end{proof}

The proof of Lemma~\ref{lem:geometric-reduction} is based on the following deterministic lemma which says that our reduction goes through if we have some basic control on the least singular value of $M^{\ast}$. 

\begin{lemma}\label{clm:update-formula}
Let $A$ be a real $n\times n$ matrix, let $Y$ be the last row of $A$, let $A^{\ast}$ be the matrix $A$ with the last row removed, let $u \in \ker(A^{\ast})$, and let $\wt{\chi}(Y) = \wt{\chi}_{A^{\ast}}(Y)$.
Let $\eps>0$ be such that $\s_{n-1}(A^{\ast})\ge \eps^{3/4}n^{-1/2}$. Then 
\begin{equation}\label{eq:update-item1}   \s_{n}(A)\le \eps n^{-1/2} \qquad \Longrightarrow \qquad |\sang{u,Y}| \le \big( 1+\eps^{1/4} \big)\eps n^{-1/2} \wt{\chi}(Y).\end{equation}
Conversely, and without the assumption that $\s_{n-1}(A^{\ast})\ge \eps^{3/4}n^{-1/2}$, we have
\begin{equation}\label{eq:update-item2}  \hspace{-4em}   \s_{n}(A) > \eps n^{-1/2} \qquad  \Longrightarrow  \qquad |\sang{u,Y}| > \eps n^{-1/2}\wt{\chi}(Y) . \end{equation}

\end{lemma}
\begin{proof} Let $\s_i = \s_i(A^{\ast})$, for $i=1,\ldots, n-1$ be the singular values of $A^{\ast}$ and let $u_i$ be the corresponding unit singular vectors. The least singular value $\s_n(A)$ is the unique\footnote{Note this solution is easily seen to be unique as the left hand side is monotonically increasing to $+\infty$ in this interval while the right hand side is monotonically decreasing from $+\infty$.} solution of  
\begin{equation}\label{eq:redux1} 1 + \sum_{i=1}^{n-1}\frac{\sang{u_i,Y}^2}{\s_{i}^2-x^2} = \frac{\sang{u,Y}^2}{x^2} \qquad \text{ for } \qquad 0 \leq x \leq \s_{n-1}. \end{equation}
For $x$ in this range, we can bound the left hand side of \eqref{eq:redux1} above and below by
\begin{equation} \label{eq:redux2}   1 + \sum_{i=1}^{n-1} \frac{\sang{u_i,Y}^2}{\s_{i}^2}  \leq 1 + \sum_{i=1}^{n-1}\frac{\sang{u_i,Y}^2}{\s_{i}^2-x^2} = 
1 + \sum_{i=1}^{n-1} \frac{\sang{u_i,Y}^2}{\s_{i}^2} \cdot \frac{1}{1-(x/\s_i)^2} . \end{equation}
Now assume $\s_{n-1}(A^{\ast}) \geq \eps^{3/4}n^{-1/2}$. If $x= \s_n(A) \leq \eps n^{-1/2}$ then \eqref{eq:redux2} is at most $(1+2\eps^{1/4})\wt{\chi}^2(Y)$, which together with \eqref{eq:redux1}, implies \eqref{eq:update-item1}. Conversely, if $x = \s_n(A) > \eps n^{-1/2}$ then \eqref{eq:update-item2} follows by putting the lower bound at \eqref{eq:redux2} together with \eqref{eq:redux1}. Note that this conclusion is independent of the hypothesis $\s_{n-1}(A^{\ast})\ge \eps^{3/4}n^{-1/2}$.
\end{proof}

To finish the proof of Lemma~\ref{lem:geometric-reduction} we need the following theorem of Rudelson and Vershynin \cite{RV09}.

\begin{theorem}\label{thm:RV2}
For all $\g \geq 0$, we have 
\[ \mb{P}\big( \s_{n-1}(M^{\ast})\le \gamma n^{-1/2} \big) \le C\gamma^2 + 2e^{-\Omega(n)},\]
where $C>0$ and the implicit constant may depend on entry distribution of $M^{\ast}$.
\end{theorem}

\begin{proof}[Proof of Lemma~\ref{lem:geometric-reduction}]
Simply note that we may assume that $\s_{n-1}(M^{\ast}) > \eps^{3/4} n^{-1/2}$ since we may apply Theorem~\ref{thm:RV2} to see $ \PP\big( \s_{n-1}(M^{\ast}) \leq \eps^{3/4}n^{-1/2} \big) = O(\delta_n\eps)$. Thus assuming $\s_{n-1}(M^{\ast}) > \eps^{3/4} n^{-1/2}$, we now apply Lemma~\ref{clm:update-formula}.\end{proof}

\section{Truncation step I: an initial reduction}\label{sec:truncation-initial}

In the next few sections we set up our main ``truncation step'', described in \eqref{eq:main-distance-to-subspace} in the proof outline. The point is that $\widetilde{\chi}^2$ can replaced with a truncated sum $\chi^2$ with extremely high probability. In this section, we prove that this truncation is possible \emph{assuming} that a ``regularity event'' $\cR$ (to be defined in just a moment) is extremely likely. Indeed, in Sections (Sections~\ref{sec:truncation-step-II} through \ref{sec:truncation-completion}) we shall show that 
\begin{equation} \PP( \s_n(M)\leq \eps n^{1/2} \wedge \cR^c ) = O(\delta_n\eps).\end{equation}
To state our truncation step, we recall that $M^{\ast}$ is $M$ with the last row removed,  that the $v_i$ are the singular unit vectors of $M^{\ast}$, that $v$ is a fixed kernel vector of $M^{\ast}$ and that \[\delta_n = (\log n)^{-c},  \qquad  \quad  \ell = \sqrt{\log n}, \qquad \text{ and  } \qquad  \chi^2(X) = \sum_{i=1}^{\ell}\frac{\sang{v_{n-i},X}^2}{\s_{n-i}(M^{\ast})^2}, \] for some sufficiently small $c>0$, as defined at \eqref{eq:def-delta-ell-chi}. The following is our main ``truncation step''.

\begin{lemma}\label{lem:main-geometric-lemma} For $e^{-cn} <  \eps < n^{-c}$ we have 
\[ \PP\big( \s_n(M) \leq \eps n^{-1/2}  \big) \leq  \PP\big( |\la v, X \ra| \leq \big( 1+\delta_n \big) \eps n^{-1/2} \chi(X) \big) + O(\delta_n\eps) \]
and conversely we have 
\[  \PP\big( \s_n(M) \leq \eps n^{-1/2}  \big) \geq  \PP\big( |\la v, X \ra| \leq \big( 1-\delta_n \big) \eps n^{-1/2} \chi(X)\big) + O(\delta_n \eps), \]
where $c>0$ is an absolute constant.
\end{lemma}

We now define our regularity event $\cR = \cR_1 \cap \cR_2\cap\cR_3 \cap \cR_4$, where $\cR_1$ is defined to be the event
\begin{equation}\label{eq:reg-prop-1}  \s_{n-1}(M^\ast)\ge(\log n)^{-3} n^{-1/2} , \end{equation}
$\cR_{2}$ is defined to be the event
\[ |\sang{v_{n-k},X}|\le\max\{ k^{1/8}, \, \log\log n \}, \text{ for all } 1 \leq k \leq n-1 \]
intersected with the event
\[ \s_{n-k}(M^\ast)\ge k^{3/4} n^{-1/2}, \text{ for all } k\ge \log\log n .\]
We then define $\cR_3$ and $\cR_4$ to be the events 
\[ \s_{n-(\log\log n)^2}(M^{\ast})\le(\log\log n)^3 n^{-1/2} \quad \text{ and } \quad  \sum_{i\leq (\log\log n)^2}\sang{v_{n-i},X}^2\ge\log\log n , \]
respectively. The major objective of the following few sections will be to show the following.

\begin{lemma} \label{lem:main-regularity-lemma} For all $\eps >0$ satisfying $e^{-cn} <\eps < n^{-c}$, we have that 
\[ \PP\big( \s_n(M) \leq \eps n^{-1/2} \wedge \cR^c \big)  = O(\delta_n \eps ), \] where $c>0$ is an absolute constant.
\end{lemma}

We now turn to show that Lemma~\ref{lem:main-regularity-lemma} implies Lemma~\ref{lem:main-geometric-lemma}.  For this, we prove the following deterministic lemma.

\begin{lemma}\label{lem:reduce-form}
Let $0 \leq \eps\le(\log n)^{-4}$ and let $A$ be an $n \times n$ matrix with $A \in \cR$. Let $A^{\ast}$ be $A$ with the last row $Y$ removed, let $u \in \ker(A^{\ast})$, and write $\chi = \chi_{A^{\ast}}$. Then 
\[ \s_{n}(A)\le\eps n^{-1/2} \quad \Longrightarrow \quad |\sang{u,Y}|\le \big(1 + \delta_n \big)\eps n^{-1/2} \cdot \chi(Y).\]
Conversely,
\[ \s_{n}(A)> \eps n^{-1/2}  \quad \Longrightarrow \quad  |\sang{u,Y}| > \big(1 - \delta_n\big)\eps n^{-1/2} \cdot \chi(Y) .\]
\end{lemma}
\begin{proof} Let $u_i$ be the unit singular vector corresponding to $\s_i(A^{\ast})$. We show that if $A \in \cR$ we have 
\begin{align} 
\wt{\chi}^2(Y) =  1 + \sum_{i=1}^{n-1}\frac{\sang{u_i,Y}^2}{\s_i(A^{\ast})^2}  = \big( 1+O(\delta_n) \big) \sum_{i=1}^{\ell} \frac{\la u_{n-i},Y\ra^2}{\sigma_{n-i}(A^{\ast})^2} = \big(1+O(\delta_n)\big) \chi^2(Y). \label{eq:simplified-form}
\end{align}
Once we have have shown this, the result follows from Lemma~\ref{clm:update-formula}. Now, since $A \in \cR \subset \cR_3 \cap \cR_4$,
\begin{equation}\label{eq:reduce-form} 1 + \sum_{i=1}^{n-1}\frac{\sang{u_i,Y}^2}{\s_i(A^{\ast})^2}
\ge \frac{1}{\s_{n-(\log\log n)^2}(A^{\ast})^{2}} \sum_{i\leq (\log\log n)^2}\sang{u_{n-i},Y}^2\ge \frac{n}{(\log\log n)^{5}}.\end{equation}
Further, since $A \in \cR \subset \cR_2$, and since $\ell = \sqrt{\log n}$, we bound the tail of the sum $\wt{\chi}$ as 
\[\sum_{i=\ell+1}^{n}\frac{\sang{u_{n-i},Y}^2}{\s_{n-i}(A^{\ast})^2}\le n \sum_{i= \ell+1 }^{n}\frac{i^{1/4}}{i^{3/2}}\lesssim n \cdot \ell^{-1/4}\lesssim \frac{n}{ (\log n)^{1/8}}.\]
This inequality combined with \eqref{eq:reduce-form} implies \eqref{eq:simplified-form} and thus completes the proof.\end{proof}

\begin{proof}[Proof of Lemma~\ref{lem:main-geometric-lemma} assuming Lemma~\ref{lem:main-regularity-lemma}]
Simply write
\[ \PP\big( \sigma_{n}(M) \leq \eps n^{-1/2} \big)  = \PP\big( \sigma_{n}(M) \leq \eps n^{-1/2} \wedge \cR^c \big)  + \PP\big( \sigma_{n}(M) \leq \eps n^{-1/2} \wedge \cR \big),\]
and use Lemma~\ref{lem:main-regularity-lemma} to deal with the first term on the right hand side and use Lemmas~\ref{lem:reduce-form} and~\ref{lem:geometric-reduction} to deal with the second. 

For the lower bound, using Lemma~\ref{lem:reduce-form} we have that
\[\PP\big( \sigma_{n}(M) \leq \eps n^{-1/2} \big) \ge \PP\big( \sigma_{n}(M) \leq \eps n^{-1/2} \wedge \cR \big) \ge \PP\big(|\sang{u,X}|\le (1-\delta_n)\eps n^{-1/2} \cdot \chi(Y) \wedge \cR) \]
which is equal to
\[  \PP\big(|\sang{u,X}|\le (1-\delta_n)\eps n^{-1/2} \cdot \chi(Y)\big) - \PP\big(|\sang{u,X}|\le (1-\delta_n)\eps n^{-1/2} \cdot \chi(X) \wedge \cR^{c}\big).\]
It therefore suffices to prove that 
\[\PP\big(|\sang{u,X}|\le (1-\delta_n)\eps n^{-1/2} \cdot \chi(X) \wedge \cR^{c}\big) = O(\delta_n\eps).\]
Recall the definition of $\wt{\chi}^2(X) = 1 + \sum_{i=1}^{n-1}\frac{\sang{v_i,X}^2}{\sigma_i(M)^2}$ and observe by definition that $\chi^2(X)\le \wt{\chi}^2(X)$. Thus it suffices to show 
\[\PP\big(|\sang{u,X}|\le (1-\delta_n)\eps n^{-1/2} \cdot \wt{\chi}(X) \wedge \cR^{c}\big) = O(\delta_n\eps).\]
However from the second item of Claim~\ref{clm:update-formula}, we have that $|\sang{u,X}|\le (1-\delta_n)\eps n^{-1/2} \cdot \wt{\chi}(X)$ implies $\sigma_{n}(M)\le (1-\delta_n)\eps n^{-1/2} \le \eps n^{-1/2}$. Therefore 
\[\PP\big(|\sang{u,X}|\le \eps n^{-1/2} \cdot \wt{\chi}(X) \wedge \cR^{c}\big) \le \PP\big(s_{n}(M)\le \eps n^{-1/2} \wedge \cR^{c}\big) = O(\delta_n \eps)\]
where the last bound follows from Lemma~\ref{lem:main-regularity-lemma}.
\end{proof}

\section{Truncation step II: preparation for the proof of Lemma~\ref{lem:main-regularity-lemma} } \label{sec:truncation-step-II}

In the next sections we prove Lemma~\ref{lem:main-regularity-lemma} and therefore complete the proof of Lemma~\ref{lem:main-geometric-lemma}. In this section we prove the following three lemmas that will be quite useful towards this goal.  

\begin{lemma}\label{lem:bootstrap}
Let $\mc{E}$ be an event which is measurable given $M^{\ast}$. Let $0<\eps\le 1$ and $\delta\in [\eps^{3/4},1]$. Then 
\[\mb{P}\big( \s_n(M)\le \eps n^{-1/2} \wedge \s_{n-1}(M^{\ast})\ge \delta n^{-1/2} \wedge \mc{E}\big) \lesssim n \big( \eps/\delta\big)\cdot \mb{P}(\mc{E}) + e^{-\Omega(n)}.\] \end{lemma}

The second main lemma in this section is a refinement of Lemma~\ref{lem:bootstrap}, under the additional event $\cE_{r}$, a cousin of $\cR_2$, which we define to be the event  
\begin{equation}\label{eq:def-E_r-I}  |\sang{X,v_{n-k}}|\le\max\{ k^{1/8},  r \}, \text{ for all } 1\le k \le n-1\end{equation} intersected with the event
  \begin{equation}\label{eq:def-E_r-II} \s_{n-k}(M^\ast)\ge k^{3/4}n^{-1/2}, \text{ for all } k\ge r. \end{equation}
In this section we also prove the following. 

\begin{lemma}\label{lem:bootstrap-with-Eell}
Let $0<\eps\le 1$, $\delta\in [\eps^{3/4},1]$, $r\in [1,n-1]$ and let $\mc{E}$ be an event which is measurable given $M^{\ast}$. Then 
\begin{equation}\label{eq:cor-Eell-bootstrap} \mb{P}\big(\s_n(M)\le \eps n^{-1/2} \wedge \s_{n-1}(M^{\ast})\ge \delta n^{-1/2} \wedge \mc{E}_{r} \wedge \mc{E}\big) \lesssim r^{3/2}  \big( \eps/\delta\big)  \cdot \mb{P}(\mc{E})+ e^{-\Omega(n)}.\end{equation}
\end{lemma}

The next lemma gives us an analogue of the above two lemmas in the case of the non-$M^{\ast}$-measurable event $\cE  =\{ \la X,w \ra > t \}$. To prove this result, we will need a decoupling result from the work of Campos, Jenssen, Michelen, and the second author \cite{CJMS24}.

\begin{lemma}\label{lem:bootstrap-w-tail}
Let $0<\eps\le 1$, $\delta\in [\eps^{3/4},1]$, $t\ge 1$ and let $w\in \mb{S}^{n-1}$ be measurable given $M^{\ast}$.  Then 
\begin{equation}\label{eq:bootstrap-w-tail-1} \mb{P}\big(\s_n(M)\le \eps n^{-1/2} \wedge \s_{n-1}(M^{\ast})\ge \delta n^{-1/2} \wedge |\sang{w,X}|\ge t\big)\lesssim n \big( \eps/\delta\big) \cdot e^{-\Omega(t^2)} + e^{-\Omega(n)}.\end{equation}
Furthermore,
\begin{equation}\label{eq:bootstrap-w-tail-2} \mb{P}\big( \s_n(M)\le \eps n^{-1/2} \wedge \s_{n-1}(M^{\ast})\ge \delta n^{-1/2} \wedge \mc{E}_{r} \wedge |\sang{w,X}|\ge t\big)\lesssim  r^{3/2} \big( \eps/\delta \big) \cdot e^{-\Omega(t^2)}  + e^{-\Omega(n)}.\end{equation}
\end{lemma}

In the remainder of this section we proceed to prove these lemmas. The first step in this direction is to prove the following geometric reductions, based on Lemma~\ref{clm:update-formula}. 

\subsection{Proof of Lemma~\ref{lem:bootstrap}}

\begin{lemma}\label{lem:bootstrapping}
Let $\eps >0$, $\delta  \in [\eps^{3/4},1]$ and let $\mc{E}$ be an event. Then
\[\mb{P}\big(\s_{n}(M)\le \eps n^{-1/2} \wedge \s_n(M^{\ast})>\delta n^{-1/2} \wedge \cE \big) \leq \mb{P}\big(|\sang{v ,X}|\le 2\big( \eps/\delta \big)(\snorm{X}_2^2 + 1)^{1/2} \wedge \cE \big).\]
\end{lemma}
\begin{proof}
Since $\delta \geq \eps^{3/4}$ we use Lemma~\ref{clm:update-formula}. We replace the event $\s_{n}(M)\le \eps n^{-1/2}$ with the event
\[ |\sang{v,X}| \le 2\eps n^{-1/2} \bigg(1+\sum_{i=1}^{n-1}\frac{\sang{v_i,X}^2}{\s_i(M^\ast)^2}\bigg)^{1/2} \leq 2\big( \eps/\delta \big) (\|X\|^2_2+1)^{1/2}, \] where the last inequality holds since $\s_{n-1}(M^{\ast}) \geq \delta n^{-1/2}$ and since the $v_i$ are orthogonal. 
\end{proof}

The proof of Lemma~\ref{lem:bootstrap} now follows quickly from Lemma~\ref{lem:bootstrapping}, along with two key results of Rudelson and Vershynin \cite{RV08}. For the first we define, for a vector $v \in \R^n$ and $\alpha,\gamma>0$, the \emph{least common denominator} of $v$ as
\begin{equation} \on{LCD}_{\alpha,\gamma}(v) = \inf\big\{\theta>0\colon\on{dist}(\theta v, \mb{Z}^n)<\min\{\gamma\snorm{\theta v}_2, \sqrt{\alpha n}\} \big\}.\end{equation}
The following fundamental result of Rudelson and Vershynin tells us that kernel vectors of random matrices have large least common denominator.

\begin{theorem}\label{thm:LCD}
There exist constants $\alpha, \gamma, c >0$, depending only $\|\xi\|_{\psi_2}$, for which 
\[\mb{P}_{M^{\ast}}\big( \exists v \in \ker(M^{\ast}) \emph{ with }\on{LCD}_{\alpha,\gamma}(v)<e^{cn} \big) \le 2e^{-cn}.\]
\end{theorem}

Since the parameters $\alpha,\gamma$ will be uninteresting for us, we simply fix them so that Theorem~\ref{thm:LCD} holds and write $\on{LCD} = \on{LCD}_{\alpha,\gamma}$ (also $c$ will be sufficiently small). 

We shall also use another fundamental result of Rudelson and Vershynin (see \cite[Theorem~6.2]{Rud14}) which tells us that vectors with  large least common denominator have good anti-concentration properties. 

\begin{theorem}\label{thm:RV-anticoncentration}
For $\eps >0$, let $u \in \R^n$, with $\|u\|_2 = 1$, and let $Y = (Y_1,\ldots,Y_n)$ be a random vector where the $Y_i$ are iid, mean $0$ and subgaussian. If  $\on{LCD}(u)\ge C/\eps$ then
\[ \mb{P}_Y\big( |\sang{u,Y}|\le \eps \big) \le C(\eps + e^{-c n}),\]
where $C,c>0$ depend only the the subgaussian constant $\|Y_i\|_{\psi_2}$.
\end{theorem}

We are now in a position to prove Lemma~\ref{lem:bootstrap}. For this, let us define the event 
\begin{equation} \label{eq:def-E-lcd} \cE_{\mr{lcd}} = \big\{ \forall v \in \ker(M^{\ast}) \text{ we have } \on{LCM}(v) > e^{cn}  \big\},\end{equation}
where the constant $c>0$ matches that in Theorem~\ref{thm:LCD}.

\begin{proof}[Proof of Lemma~\ref{lem:bootstrap}]
Note that $\mb{P}(\snorm{X}_2^2\ge n \log n)\le \exp(-\Omega(n\log n))$, since $X$ is subgaussian. Also Theorem~\ref{thm:LCD} tells us that $\PP(\cE_{\mr{lcd}}^c) \leq e^{-cn}$. Applying these observations and Lemma~\ref{lem:bootstrapping} gives 
\[ \mb{P}\big(\s_{n}(M)\le \eps n^{-1/2} \wedge \s_n(M^{\ast})>\delta n^{-1/2} \wedge \cE \big) \leq \mb{P}\big(|\sang{v,X}|\le 4\big( \eps/\delta \big)(n\log n)^{1/2} \wedge \cE  \wedge \cE_{\mr{lcd}}\big) + e^{-\Omega(n)}.\]
Now using the fact that $\cE_{\mr{lcd}}$ and $\cE$ depend only on $M^{\ast}$, we have that the above is
\[ \leq \max_{M^{\ast} \in \cE\cap \cE_{\mr{lcd}}}\mb{P}_X\big(|\sang{v,X}|\le 4\big( \eps/\delta \big)(n\log n)^{1/2}\, \vert\, M^{\ast} \big)\cdot \PP_{M_{\ast}}( \cE ) + e^{-\Omega(n)} \leq n(\eps/\delta )\cdot \PP(\cE) + e^{-\Omega(n)}, \]
where the last inequality follows from an application of Theorem~\ref{thm:RV-anticoncentration}.
\end{proof}

\subsection{Proof of Lemma~\ref{lem:bootstrap-with-Eell}}

We now modify the proof of the Lemma~\ref{lem:bootstrapping} to prove the refined Lemma~\ref{lem:bootstrap-with-Eell}.

\begin{lemma}\label{lem:bootstrap-2}
Let $\eps >0$, $\delta \in [\eps^{3/4},1]$, $r\in [1,n-1]$ and $\cE$ be an event. Then
\begin{align*}
\mb{P}\big(\s_{n}(M)\le \eps n^{-1/2} \wedge \s_{n-1}(M^{\ast})&\ge \delta n^{-1/2} \wedge \mc{E}_r^\ast \wedge \mc{E} \big)\\
&\le
\mb{P}\big( |\sang{v,X}|\le Cr^{3/2} \big(\eps/\delta \big)  \wedge \s_{n-1}(M^{\ast})\ge \delta n^{-1/2} \wedge \cE \big),
\end{align*}
where $C>0$ is an absolute constant. 
\end{lemma}
\begin{proof}
Since $\delta > \eps^{3/4}$ we use Lemma~\ref{clm:update-formula} to 
see that  
\[ \s_n(M) \le \eps n^{-1/2} \qquad  \Longrightarrow \qquad |\sang{v,X}| \le  2\eps n^{-1/2} \bigg(1+\sum_{i=1}^{n-1}\frac{\sang{v_i,X}^2}{\s_i(M^\ast)^2}\bigg)^{1/2}. \]
Now the sum on the right hand side, under the event $\cE^{\ast}_{r}$ and given $\sigma_{n-1}(M^\ast)\ge\delta n^{-1/2}$, is  
\[ 1+\sum_{i=1}^{n-1}\frac{\sang{v_i,X}^2}{\s_i(M^\ast)^2} \leq
\sum_{i=1}^{r}\frac{r^2}{\s_{n-i}(M^\ast)^2}+\sum_{i>r}\frac{(r+(i+1)^{1/8})^2}{(i^{3/4}/\sqrt{n})^2}  \leq \frac{C n r^{3}}{\delta^{2}}.\]
Thus we may replace $\sigma_n(M) \leq \eps n^{-1/2}$ by $|\la v ,X \ra| \leq C r^{3/2} \eps/\delta$ in the probability.
\end{proof}

The proof of Lemma~\ref{lem:bootstrap-with-Eell} now follows in the same manner as the proof of Lemma~\ref{lem:bootstrap}.

\begin{proof}[Proof of Lemma~\ref{lem:bootstrap-with-Eell}]
We apply Theorem~\ref{thm:LCD}, then Lemma~\ref{lem:bootstrap-2} and then condition on $M^{\ast}$, as in the proof of Lemma~\ref{lem:bootstrap}, to see the left hand side of \eqref{eq:cor-Eell-bootstrap} is
\[ \leq \max_{M^{\ast} \in \cE\cap \cE_{\mr{lcd}}}\mb{P}_X\big(|\sang{v,X}|\le 4r^{3/2}\big( \eps/\delta \big)\, \vert\, M^{\ast} \big) \cdot \PP_{M_{\ast}}( \cE ) + e^{-\Omega(n)} \lesssim r^{3/2}(\eps/\delta )\cdot \PP(\cE) + e^{-\Omega(n)}, \]
where the last inequality follows from an application of Theorem~\ref{thm:RV-anticoncentration}.\end{proof}

\subsection{Proof of Lemma~\ref{lem:bootstrap-w-tail}}
We now turn to prove one more variant of the above lemmas. Again the proof is similar to that of Lemma~\ref{lem:bootstrap}, but this time we incorporate a tool developed by Campos, Jenssen, Michelen, and the second author \cite{CJMS24}, which tells us that ``anti-concentration'' events are approximately negatively correlated with ``large-deviation'' events. 

\begin{theorem}\label{thm:decor}
For $\eps > 0$, let $u, w \in \R^n$ have $\|u\|_2 = \|w\|_2 = 1$ and let $Y = (Y_1,\ldots,Y_n)$ be a random vector, where the $Y_i$ are iid, mean $0$ and subgaussian. If $t > 0$ and  $\on{LCD}(u)\ge C/\eps$  then 
\[ \mb{P}\big(|\sang{u,Y}|\le \eps \wedge \sang{w,Y}>t \big) \le C\eps \cdot e^{-ct^2} + e^{-c(n + t^2)}, \]
where $C,c>0$ depend only on the subgaussian constant $\| Y_i\|_{\psi_2}$.
\end{theorem}

We now prove Lemma~\ref{lem:bootstrap-w-tail}.

\begin{proof}[Proof of Lemma~\ref{lem:bootstrap-w-tail}] Here there are two items to prove. For \eqref{eq:bootstrap-w-tail-1}, apply Theorem~\ref{thm:LCD} to deal with $\cE_{\mr{lcd}}$ then apply Lemma~\ref{lem:bootstrapping} with $\cE = \{ \la v_n,X\ra > t \}$ to the left hand side of \eqref{eq:bootstrap-w-tail-1} to see it is
\[ \leq \mb{P}\big(|\sang{v,X}|\le 4\big( \eps/\delta \big)(n\log n)^{1/2} \wedge \cE  \wedge \cE_{\mr{lcd}}\big) + e^{-\Omega(n)}.\]
Conditioning on the worst $M^{\ast} \in \cE_{\mr{lcd}}$, we see, using Theorem~\ref{thm:decor}, the above is at most
\[ \max_{v, w }\, \PP_X\big(|\sang{v,X}|\le 4\big( \eps/\delta \big)(n\log n)^{1/2} \wedge \la w, X \ra > t \big) \lesssim n \big( \eps/\delta \big) e^{-ct^2} + e^{-\Omega(n)} ,\]
where the maximum is taken over all $v,w$ with $\|v\|_2 = \|w\|_2 = 1$ and $\on{LCD}(v) > e^{cn}$. This proves \eqref{eq:bootstrap-w-tail-1}. The proof of \eqref{eq:bootstrap-w-tail-2} is the same with Lemma~\ref{lem:bootstrap-2} applied in place of Lemma~\ref{lem:bootstrapping}.
\end{proof}

\section{Truncation step III:  taking care of the events \texorpdfstring{$\cR_1$ and $\cR_2$}{}}
\label{sec:truncation-R1R2}
The goal of this section will be to prove the first ``half'' of our main regularity lemma, Lemma~\ref{lem:main-regularity-lemma}. For this, recall that $M^{\ast}$ is the matrix $M$ with the last row removed and that the $v_i$ are the unit singular vectors of $M^{\ast}$. Recall that $\delta_n = (\log n)^{-c}$ for some small $c>0$, as defined at \eqref{eq:def-delta-ell-chi}, 
\[ \cR_1 = \big\{ \s_{n-1}(M^\ast)\ge(\log n)^{-3} n^{-1/2} \big\}  \] and that $\cR_2$ is the event \[ |\sang{v_{n-k},X}|\le\max\{ k^{1/8}, \, \log\log n \}, \text{ for all } 1 \leq k \leq n-1 , \]
intersected with the event
\[ \s_{n-k}(M^\ast)\ge k^{3/4} n^{-1/2}, \text{ for all } k\ge \log\log n .\]
The following lemma is the first half of our main regularity lemma. 

\begin{lemma}\label{lem:master-reg-first-half} For all $e^{-cn} < \eps < n^{-c}$ we have 
\[ \PP\big( \s_n(M) \leq \eps n^{-1/2} \wedge (\cR_1 \wedge \cR_2)^{c} \big) = O(\delta_n\eps), \] 
where $c>0$ is an absolute constant.
\end{lemma}

We prove Lemma~\ref{lem:master-reg-first-half} by first proving a weaker analogue of it which we then ``bootstrap'' to the full result.

\subsection{A weaker version of Lemma~\ref{lem:master-reg-first-half}}
We apply Theorem~\ref{thm:RV2} of Rudelson and Vershynin, discussed in Section~\ref{sec:geometric-reduction}, which says that  
\begin{equation}\label{eq:RV-statement2} \mb{P}\big( \s_{n-1}(M^{\ast})\le \gamma n^{-1/2} \big) \le C\gamma^2 + 2e^{-\Omega(n)},\end{equation} Where $C>0$ is a constant depending only on the subgaussian constant of the underlying random variable. 

\begin{lemma}\label{lem:second-sing-val-weak} For $e^{-cn} < \eps < n^{-c}$ we have 
\begin{equation} \label{eq:second-sing-val-weak}\PP\big( \sigma_n(M) \leq \eps n^{-1/2}  \wedge   \s_{n-1}(M^{\ast}) \leq n^{-5/2} \big) = O(\delta_n \eps), \end{equation}
where $c>0$ is an absolute constant.
\end{lemma}
\begin{proof}
From Theorem~\ref{thm:RV2}, we see the probability that $\s_{n-1}(M^{\ast}) \leq \eps^{3/4} n^{-1/2}$ is $O(\eps^{3/2})$. Thus the left hand side of \eqref{eq:second-sing-val-weak} is at most 
\begin{equation} \label{eq:sing-value-cut} \mb{P}\big( \s_{n}(M)\le \eps n^{-1/2} \wedge \s_{n-1}(M^{\ast})\in[\eps^{3/4}n^{-1/2} ,n^{-5/2}]\big) + O(\delta_n \eps).  \end{equation} 
We now dyadically split up the interval $[\eps^{3/4}n^{-1/2} ,n^{-5/2}]$ by setting $\g_j = 2^{j}\cdot \eps^{3/4} $ for all $j\ge 0$ such that $\g_jn^{-1/2}\le n^{-5/2}$. We now put $\mc{E}_j=\{\s_{n-1}(M^\ast)\le 2\g_j n^{-1/2} \}$. Thus, summing over $j$, \eqref{eq:sing-value-cut} is 
\[ \leq \sum \mb{P}\big( \s_{n}(M)\le \eps n^{-1/2} \wedge \s_{n-1}(M^{\ast}) \geq \g_j n^{-1/2} \wedge\mc{E}_j \big)  + O(\delta_n \eps). \]
We now apply Lemma~\ref{lem:bootstrap} to each summand to see the above is at most
\[ \sum n \big(\eps/\g_j\big) \PP( \cE_j )  + O(\delta_n \eps) \leq \sum n \big(\eps/\g_j\big) \cdot \big(\g_j^2 + e^{-\Omega(n)} \big) + o(\eps) \leq \eps/n + O(\delta_n \eps) = O(\delta_n \eps), \] as desired. Here we applied Theorem~\ref{thm:RV2} for the first inequality. \end{proof}

We now prove that a weaker version of the regularity event $\cR_2$ fails with probability $o(\eps)$, under the event $\sigma_n(M)\leq \eps n^{-1/2}$. To define this event, recall the definition of $\cE_{r}$, which we defined at \eqref{eq:def-E_r-I} and \eqref{eq:def-E_r-II}. We define 
our ``weak'' version $\wt{\cR}_2$ of $\cR_2$ to be
\begin{equation}\label{eq:def-wtR2} \wt{\cR}_2 = \cE_{\log n}. \end{equation}
To handle $\wt{\cR}_2$ need the following control on the lower tails of $\s_i(M^{\ast})$. The following theorem is an immediate consequence (using interlacing) of the work of Nguyen \cite[Theorem~1.4]{Ngu18}.

\begin{theorem}\label{thm:sing-val-lower-tail} For all $k\geq 1$, we have 
\[ \PP( \sigma_{n-k}(M^{\ast}) < c k \cdot n^{-1/2} \big) \leq e^{-k^2/8} + 2e^{-cn},\]  where $c>0$ is a constant depending only on $\|\xi\|_{\psi_2}$. 
\end{theorem}

\begin{lemma}\label{lem:cEstar-weak}
For $e^{-cn} < \eps < n^{-c}$ we have 
 \begin{equation}\label{eq:cEstar-logn} \mb{P}\big( \s_{n}(M)\le \eps n^{-1/2} \wedge \wt{\mc{R}}_2^c \big) = O(\delta_n \eps), \end{equation} where $c>0$ is an absolute constant. 
\end{lemma}

\begin{proof} By Lemma~\ref{lem:second-sing-val-weak}, we may intersect the right hand side of \eqref{eq:cEstar-logn} with the event $\sigma_{n-1}(M^{\ast}) > n^{-5/2}$ at a loss of $O(\delta_n\eps)$ in probability. Set
$r_i = \max\{(n-i+1)^{1/8},\log n \}$ and note that, by Lemma~\ref{lem:bootstrap-w-tail},  
\begin{equation}\label{eq:Estar-weak-1} \mb{P}\big( \s_{n}(M)\le \eps n^{-1/2} \wedge \s_{n-1}(M^{\ast}) > n^{-5/2} \wedge |\la v_i, X \ra| > r_i \big)  \lesssim n^3 \eps \cdot e^{-\Omega(r_i^2) } + e^{-\Omega(n)}, \end{equation} which is at most $\eps n^{-\omega(1)}$. 

Now note that if $\cF_k = \{  \sigma_{n-k}(M^{\ast}) < k^{3/4}n^{-1/2}  \}$, from Lemma~\ref{lem:bootstrap}, we have 
\begin{equation}\label{eq:Estar-weak-2}  \mb{P}\big( \s_{n}(M)\le \eps n^{-1/2} \wedge \s_{n-1}(M^{\ast})> n^{-5/2} \wedge \cF_k  \big) \lesssim n^3 \eps \cdot \PP(\cF_k)  + e^{-\Omega(n)} \leq \eps n^{-\omega(1)},\end{equation}
 where the last inequality holds by Theorem~\ref{thm:sing-val-lower-tail}, when $k \geq \log n$, as in the definition of the $\wt{\cR}_2$.

Now observe that the event on the left hand side of \eqref{eq:cEstar-logn} is the union of $n$ events of the type found on the left and side of \eqref{eq:Estar-weak-1} and of the type found on the left hand side of \eqref{eq:Estar-weak-2}, with $k \geq \log n$. Thus we can union bound over \eqref{eq:Estar-weak-1} and \eqref{eq:Estar-weak-2} to complete the proof of the Lemma. \end{proof}

\subsection{Bootstrapping to the proof of Lemma~\ref{lem:master-reg-first-half}}

We can now bootstrap these results to prove Lemma~\ref{lem:master-reg-first-half}, which follows from combining Lemma~\ref{lem:second-sing-val} and Lemma~\ref{lem:r2-alone}. In both cases the proof closely follows the proof of the ``weaker'' counterpart.

\begin{lemma}\label{lem:second-sing-val} For $e^{-cn} < \eps < n^{-c}$ we have 
\begin{equation}\label{eq:second-sing-val} \PP\big( \sigma(M) \leq \eps n^{-1/2}  \wedge  \s_{n-1}(M^\ast) < (\log n)^{-3} n^{-1/2}  \big) = O(\delta_n\eps), \end{equation}
where $c>0$ is an absolute constant.
\end{lemma}
\begin{proof}
Apply Theorem~\ref{thm:RV2}, Lemma~\ref{lem:second-sing-val-weak}, and Lemma~\ref{lem:cEstar-weak} to see that we may intersect the event in the left hand side of \eqref{eq:second-sing-val} with the events $\s_{n-1}(M^{\ast})\ge \eps^{3/4}$,  $\s_{n-1}(M^{\ast})\ge n^{-5/2}$, and $\wt{\mc{R}}_{2}$, at the loss of $O(\delta_n\eps)$ in probability. Thus we are interested in the event 
\[ \sigma_{n-1}(M^{\ast}) \in [\max\{\eps^{3/4},n^{-5/2}\}, (\log n)^{-3} n^{-1/2}]. \]
Again we dyadically partition this interval, by setting $\g_j n^{-1/2} = 2^{j} \cdot \max\{ \eps^{3/4},n^{-5/2} \}$ for all $j\ge 0$ such that $\g_j\le (\log n)^{-3}$. Now put $\cF_j = \{  \s_{n-1}(M^{\ast}) \leq 2\g_jn^{-1/2}  \}$ so the probability in \eqref{eq:second-sing-val} is
\[\leq \sum_{j}\mb{P}\big(\s_{n}(M)\le \eps n^{-1/2} \wedge \s_{n-1}(M^{\ast}) \geq \g_jn^{-1/2} \wedge \wt{\mc{R}_2} \wedge \cF_j \big) + O(\delta_n\eps).\]
Thus, recalling that $\wt{\cR}_2=\cE_{\log n}$ (defined at \eqref{eq:def-wtR2}), we apply Lemma~\ref{lem:bootstrap-with-Eell}  to each summand with $r  = \log n$ to see that the above is at most
\[ \lesssim (\log n)^{3/2} \sum \big( \eps/\g_j \big)
 \mb{P}( \s_{n-1}(M^{\ast}) \leq 2\g_jn^{-1/2} \big)   + O(\delta_n\eps) \lesssim  (\log n)^{3/2} \eps \sum \g_j  + O(\delta_n\eps),  \] where we applied Theorem~\ref{thm:RV2} to the probability in each summand. This is at most $O(\delta_n \eps)$.\end{proof}

We now note that $\{ \s_n(M) \leq \eps n^{-1/2} \} \wedge \cR_2^c $ occurs with negligible probability. 

\begin{lemma}\label{lem:r2-alone} For $e^{-cn} < \eps < n^{-c}$ we have 
\[\PP( \sigma_n(M) \leq \eps n^{-1/2} \wedge  \cR_2^c \big) = O(\delta_n\eps) ,\] 
where $c>0$ is an absolute constant.
\end{lemma}
\begin{proof}We follow the proof of Lemma~\ref{lem:cEstar-weak}, but use Lemma~\ref{lem:second-sing-val} in place of Lemma~\ref{lem:second-sing-val-weak}. We note that instead of taking a union bound over $n$ events, we only need to consider at most $2(\log n)^8$ events not covered by $\wt{R}_2$.
\end{proof}

\section{Truncation step IV: Taking care of the events \texorpdfstring{$\cR_3$ and $\cR_4$}{}} \label{sec:truncation-completion}

Our goal in this section is to prove the following lemma which amounts to the ``second half'' of Lemma~\ref{lem:main-regularity-lemma}. This, when taken together with the work in Section~\ref{sec:truncation-initial}, completes the proof of our ``truncation step'', Lemma~\ref{lem:main-geometric-lemma}.

\begin{lemma}\label{lem:truncation2} For $e^{-cn} < \eps < n^{-c}$ we have 
\[ \PP\big( \sigma_n(M)\leq \eps n^{-1/2} \wedge ( \cR_3 \wedge \cR_4)^c \big) = O(\eps \delta_n ), \]
where $c>0$ is an absolute constant. 
\end{lemma}
Recall again that $M^{\ast}$ is $M$ with the last row removed, and that $v_i$ are the unit singular vectors corresponding to the singular values $\sigma_i(M^{\ast})$. Recall we defined $\cR_3$ and $\cR_4$ to be the events (respectively)
\[ \s_{n-(\log\log n)^2}(M^{\ast})\le(\log\log n)^3 n^{-1/2} \quad \text{ and } \quad  \sum_{i\leq (\log\log n)^2}\sang{v_{n-i},X}^2\ge\log\log n . \]

\subsection{Dealing with \texorpdfstring{$\cR_3$}{R3}} To ensure that we can assume $\cR_3$, we need the following result on the upper tails of the $\sigma_i(M^{\ast})$.

\begin{lemma}\label{lem:uppertail-smallsingvalues} 
Let $k\le n^{c}$, and $t\ge C$ then
\[\mb{P}\big(\s_{n-k}(M^{\ast})\ge tk \cdot n^{-1/2} \big) 
\lesssim e^{-(tk)^2/C} + n^{-c},\]
where $C,c>0$ depend only on the subgaussian constant. 
\end{lemma}

This result concerns only the ``macroscopic'' properties of the spectrum and thus falls into a category of statements that we now have a good understanding of. Indeed, this result can be extracted from combining the work of Szarek \cite{Sza91} and Tao and Vu \cite{TV10}.

We now deal with $\cR_3$ and, since it is convenient, we also deal with a related event that we will use in Section~\ref{sec:replacement}. 

\begin{lemma}\label{lem:R3} For $e^{-cn} < \eps < n^{-c} $ we have 
\begin{equation}\label{eq:R3} \mb{P}\big( \s_{n}(M)\le \eps n^{-1/2} \wedge \cR_3^c \big) = O(\delta_n\eps). \end{equation}
and
\begin{equation}\label{eq:R3-related} \PP\big(\s_{n}(M)\le \eps n^{-1/2} \wedge \s_{n-\sqrt{\log n}}(M^\ast)\ge (\log n)n^{-1/2} \big) = O(\delta_n\eps),\end{equation}
where $c>0$ is an absolute constant. 
\end{lemma}
\begin{proof} Here we prove \eqref{eq:R3} and simply note that the proof of \eqref{eq:R3-related} is almost identical. From Lemma~\ref{lem:cEstar-weak} and Lemma~\ref{lem:second-sing-val} we intersect with the events $\wt{\cR}_2$ and $\sigma_{n-1}(M^{\ast}) \geq (\log n)^{-3} n^{-1/2}$ at a loss of $O(\delta_n\eps)$ in probability. Now apply Lemma~\ref{lem:bootstrap-with-Eell} with $\cE =  \cR^c_3$ to see
\[ \PP\big(\s_{n}(M)\le \eps n^{-1/2} \wedge \s_{n-1}(M^{\ast}) \geq (\log n)^{-3}n^{-1/2} \wedge \cE_{\log n} \wedge \cR_3^c \big)\leq (\log n)^{9/2} \eps \cdot \PP(\cR_3^c) + e^{-\Omega(n)} . \]
which is $O(\delta_n\eps)$, since we may apply Lemma~\ref{lem:uppertail-smallsingvalues} to see
\[ \PP(R_3^c) = \PP\big(\sigma_{n-\log\log n}(M^{\ast}) \geq (\log \log n)^3 n^{-1/2} \big) \leq (\log n)^{-\omega(1)}.\qedhere \] \end{proof}

\subsection{Dealing with \texorpdfstring{$\cR_4$}{R4}} For this we use a simpler variant of a negative correlation theorem due to Campos, Jenssen, Michelen and the second author. 

\begin{theorem}\label{thm:decor-harder} Let $u \in \R^n$ satisfy $\|u\|_2 =1$ and let $w_1,\ldots, w_k \in \R^n$ be orthogonal unit vectors. Furthermore let $Y = (Y_1,\ldots,Y_n)$
be a random vector where the $Y_i$ are iid, mean $0$ and subgaussian. 

If $\on{LCD}(u) \geq C/\eps$ and $k\cdot \snorm{u}_{\infty} + \sum_{i=1}^{k}\snorm{w_i}_{\infty}\le (\log n)^{-3}$ then 
\[\mb{P}\bigg( |\sang{u,Y}|\le \eps\, \wedge\,  \sum_{j=1}^{k}|\sang{u_i,Y}|^2\le ck\bigg)\le C(\eps e^{-ck} + e^{-c n}),\]
where $C,c>0$ depend only on the subgaussian constant $\|Y_i\|_{\psi_2}$. \end{theorem}

This theorem can be thought of as a weaker version of Theorem 1.2 in the paper of Campos, Jenssen, Michelen, and the second author \cite{CJMS24}. Interestingly, this weaker theorem admits a significantly simpler proof (which we give in Appendix~\ref{app:negative-dependence}) while retaining much of the applicability of the original.  

We also need the following lemma that we will use to say that the unit vectors corresponding to the smallest singular values are relatively flat. 

\begin{lemma}\label{lem:unstructured-vector}
We have 
\[ \PP_{M^{\ast}}\big( \exists v \in \R^n \emph{ with } \|v \|_{2} = 1,  \|v\|_{\infty} > Cn^{-c}, \emph{ and } \|M^{\ast}v \|_2 < 1 \big) < n^{-\omega(1)},\]
where $c>0$ is absolute and $C>0$ depends only on the subgaussian constant. 
\end{lemma}
As the proof of this lemma is based on what, today, are somewhat standard techniques, we defer the proof to Appendix~\ref{app:flat-v_i}.

For the following, we recall the event 
\[ \cE_{\mr{lcd}} = \{ \exists v \in \ker(M^{\ast})\colon \on{LCD}(v) > e^{cn} \},\] which we defined at \eqref{eq:def-E-lcd}, and recall that $\PP(\cE_{\mr{lcd}}^c) < e^{-cn}$ by Theorem~\ref{thm:LCD}. Here $c>0$ denotes some sufficiently small but absolute quantity. 

\begin{lemma}\label{lem:cE_4}
For $e^{-cn} < \eps < n^{-c}$, we have  
    \begin{equation} \label{eq:lem-E4} \mb{P}\big( \s_{n}(M)\le \eps n^{-1/2} \wedge \cR_4^c \big) = O(\eps \delta_n ) ,\end{equation} where $c>0$ is an absolute constant. 
\end{lemma}
\begin{proof} 
Let $\cE_{\on{flat}}$ denote the event in Lemma~\ref{lem:unstructured-vector}. By Lemma~\ref{lem:cEstar-weak}, Lemma~\ref{lem:second-sing-val}, Lemma~\ref{lem:R3}, Theorem~\ref{thm:LCD} and Lemma~\ref{lem:unstructured-vector}
we may intersect with the events
\begin{equation}\label{eq:events-total}
\wt{\cR}_2,~ \s_{n-1}(M^{\ast})\ge (\log n)^{-3}n^{-1/2},~ \s_{n-(\log\log n)^2}(M^{\ast})\le(\log\log n)^3 n^{-1/2},~\mc{E}_{\mr{lcd}}, ~\text{ and } \mc{E}_{\on{flat}}, 
\end{equation}
with only $O(\delta_n \eps)$ loss in probability. We briefly elaborate on the case of $\mc{E}_{\on{flat}}$; the remaining are simpler. Observe that 
\begin{align*}
\mb{P}\big( \s_{n}(M)\le \eps n^{-1/2} \wedge \mc{E}_{\on{flat}}^{c})\le &\mb{P}\big( \s_{n}(M)\le \eps n^{-1/2} \wedge \s_{n-1}(M^{\ast})\ge (\log n)^{-3}n^{-1/2} \wedge \mc{E}_{\on{flat}}^{c}) \\
 + &\mb{P}\big( \s_{n}(M)\le \eps n^{-1/2} \wedge \s_{n-1}(M^{\ast})< (\log n)^{-3}n^{-1/2}) =  O(\delta_n\eps).
\end{align*}
Here the second probability is bounded by Lemma~\ref{lem:second-sing-val} while the first is bounded by combining Lemma~\ref{lem:unstructured-vector} and Lemma~\ref{lem:bootstrap}. We now let $\mc{E}$ be the intersection of the events in \eqref{eq:events-total} and let $\cE'$ be the intersection of all but the first event. Note that $\cE'$ is $M^{\ast}$ measurable. 

Now apply Lemma~\ref{lem:bootstrap-2} to see that the left hand side of \eqref{eq:lem-E4} is at most
\begin{equation}\label{eq:pre-decop} \E_{M^{\ast}} \1(M^{\ast} \in \cE') \cdot \mb{P}_X \bigg(|\sang{v,X}|\le \eps (\log n)^{10}\hspace{1.5mm}   \wedge \sum_{i \leq (\log\log n)^2}\sang{v_{n-i},X}^2\le\log\log n \bigg)  + O(\delta_n \eps). \end{equation}
We now look to apply Theorem~\ref{thm:decor-harder} to deal with the probability in \eqref{eq:pre-decop}. For this, note that 1) $\on{LCD}(v) \geq C/\eps$ since $M^{\ast} \in \cE_{\mr{lcd}}$, and 2) we have  $\|v\|_{\infty}, \|v_{n-1}\|_{\infty}, \ldots , \|v_{n-(\log\log n)^2}\|_{\infty} \leq n^{-\Omega(1)}$,
which holds by the fact that $ \| Mv \|_2,  \| Mv_{n-1} \|_2,\ldots,  \| Mv_{n-\log\log n} \|_2 \leq (\log\log n)n^{-1/2} < 1$ and the assumption that $M^{\ast} \in \cE_{\on{flat}}$. Thus we apply Theorem~\ref{thm:decor-harder} to see \eqref{eq:pre-decop} is
\[ \lesssim \eps  (\log n)^{10} \cdot e^{-\Omega((\log\log n)^2)} + e^{-\Omega(n)} = O(\delta_n \eps). \qedhere \]\end{proof}

Thus Lemma~\ref{lem:R3} and Lemma~\ref{lem:cE_4}, taken together, imply Lemma~\ref{lem:truncation2}, which, in turn, taken together with Lemma~\ref{lem:master-reg-first-half} implies Lemma~\ref{lem:main-regularity-lemma}, which completes our project of the last few sections to prove Lemma~\ref{lem:main-geometric-lemma}.

\section{A Gaussian replacement step}\label{sec:replacement}
In this section we show that we can ``replace'', without much loss, the random vector $X$ with the a Gaussian random vector $Z = (Z_1,\ldots,Z_n)$ in our quantities of interest.  Throughout this section we will think of the $ (n-1) \times n$ matrix $M^{\ast}$ as \emph{fixed} with $M^{\ast} \in \cE^{\ast}$, where $\cE^{\ast}$ is a ``quasi-randomness'' event (to be defined shortly) that fails with negligible probability. As before, we let the vectors $v_i$ be the unit vectors corresponding to the singular values $\sigma_{i}(M^{\ast})$. Let us also recall that 
\[ \chi^2(Y) = \sum_{i=1}^{\ell}\frac{\sang{v_{n-i},Y}^2}{\s_{n-i}(M^{\ast})^2} \qquad \text{ for } \qquad Y \in \R^{n}, \]
where $\ell = \sqrt{\log n}$.
Our main objective in this section is to prove the following.

\begin{lemma}\label{lem:G-replacement}
Let $e^{-cn} < \eps  < n^{-c}$, let $M^{\ast} \in \cE^{\ast}$, let $v \in \ker(M^{\ast}) \cap \bS^{n-1}$ and $\chi = \chi_{M^{\ast}}$. Then 
\begin{equation}\label{eq:G-replacement-ub}  \PP_X\big( |\sang{v,X}| \leq \eps n^{-1/2} \cdot \chi(X) \big) \leq \PP_Z\big( |\sang{v,Z}| \leq (1+ \delta_n )\cdot \eps n^{-1/2} \cdot \chi(Z) \big) + O(\delta_n\eps) \end{equation}
and 
\begin{equation}\label{eq:G-replacement-lb}   \PP_X\big( |\sang{v,X}| \leq \eps n^{-1/2} \cdot \chi(X) \big) \geq \PP_Z\big( |\sang{v,Z}| \leq (1- \delta_n )\cdot \eps n^{-1/2} \cdot \chi(Z) \big) - O(\delta_n\eps).\end{equation} Here $c>0$ is an absolute constant.
\end{lemma}

The proofs of \eqref{eq:G-replacement-ub} and \eqref{eq:G-replacement-lb} are similar to each other and are proved in two steps. To show \eqref{eq:G-replacement-ub}, we let $\cF(X) = \cF_{M^{\ast}}(X)$ be the event on the left hand side of \eqref{eq:G-replacement-ub} and let $\cG(Z) = \cG_{M^{\ast}}(Z)$ be the event on the right hand side. We first construct a pair of (very similar) smooth bump functions $f^{+},f^-$ for which
\[  \PP_X\big( \cF(X) \big) \leq \E_X\, f^{+}(X) +O(\delta_n\eps) \qquad \text{ and } \qquad \E_Z\, f^{-}(Z) \leq \PP_Z\big( \cG(Z) \big) +O(\delta_n\eps) \]
and for which $| \E_Z\,\big( f^{+}(Z)- f^{-}(Z) \big) | = O(\delta_n\eps)$. We then show, in the most subtle of the steps, that we can replace $X$ with $Z$ in the context of this smooth bump function
\begin{equation}\label{eq:replacement-step} \big|\E_X\, f^{+}(X) - \E_Z\, f^{+}(Z) \big| = O(\delta_n\eps). \end{equation} These together imply \eqref{eq:G-replacement-ub}. The proof of the lower bound at \eqref{eq:G-replacement-lb} is similar.  

\subsection{Definition of \texorpdfstring{$\cE^{\ast}$}{E*}}
This ``microscopic'' replacement is possible at these exponentially small probability scales as a consequence of several quasirandomness properties that we may assume of $M^{\ast}$ and which we isolate here. We define $\cE^{\ast}$ to be the intersection of the events 
\[  \s_{n-1}(M^{\ast})\ge(\log n)^{-3}n^{-1/2} \quad \text{ and }  \quad \s_{n-k}(M^{\ast})\ge k^{3/4}n^{-1/2} \text{ for all } k\ge\log\log n.\]
along with the events 
\begin{equation}\label{eq:def-Estar-1} \hspace{-2em} \s_{n-(\log\log n)^2}(M^{\ast})\le (\log\log n)^3n^{-1/2} \quad \text{ and } \quad \s_{n-\sqrt{\log n}}(M^\ast)\le (\log n) n^{-1/2}    \end{equation} intersected with the events $\cE_{\mr{lcd}}$ and $\cE_{\on{flat}}'$ , which are (respectively)
\begin{equation}\label{eq:def-Estar-2}  \on{LCD}(v) > e^{cn}  \qquad \text{ and } \qquad  \|v\|_{\infty}, \|v_{n-1}\|,\ldots,\|v_{n-\ell}\|_{\infty} < n^{-\Omega(1)},  \end{equation}
where $c>0$ is a sufficiently small but absolute constant. From our work in earlier sections, it is not hard to see the following.

\begin{lemma}\label{lem:E-star-negligible} For $e^{-cn} < \eps < n^{-c}$ we have 
\[ \PP\big( \s_n(M) \leq \eps n^{-1/2} \wedge (\cE^{\ast})^c \big) = O(\delta_n\eps), \]
where $c>0$ is an absolute constant. 
\end{lemma}

\subsection{Approximation by smooth bump functions}
For $M^{\ast}$ fixed we note that the event $\cF(X) = \cF_{M^{\ast}}(X)$ is defined in terms of a smooth function of the random vector 
\[ H = \big(\sang{v,X},\sang{v_{n-1},X},\ldots, \sang{v_{n-\ell},X} \big),\] while the right hand side is a smooth function of the random vector 
\[ K = \big( \sang{v,G},\sang{v_{n-1},G},\ldots, \sang{v_{n-\ell},G} \big),\] where $G \sim \mc{N}(0,1)^{\otimes n}$. Note that in fact $K \sim \mc{N}(0,1)^{\otimes (\ell+1)}$, since $v_n,\ldots,v_{n-\ell}$ are orthonormal. Thus we can think of the probabilities in Lemma~\ref{lem:G-replacement} as expectations of the indicators $\1(H \in S)$ and $\1(K \in S')$ for some regions $S,S' \subset \R^{\ell+1}$. Before we do the replacement we find smooth bump functions to approximate these regions, which will be amenable to the Fourier methods that follow. 

\begin{lemma}\label{lem:bump-function}
Let $M^{\ast}\in \mc{E}^{\ast}$. There exist functions $f^{-},f^{+} \colon \R^{\ell+1} \rightarrow \R$ for which \[ | \E_G\, f^{+}(G) -  f^{-}(G)  | = O(\delta_n\eps) \] and for $f \in \{f^+,f^{-}\}$ we have 
\[  \PP_X( \cF(X) ) \leq \E_X\, f(H) +O(\delta_n\eps) \qquad \emph{ and } \qquad \E_G\, f(K) \leq \PP_G( \cF(G) ) +O(\delta_n\eps), \]
and for all $\theta \in \R$, $\xi = (\xi_1,\ldots,\xi_{\ell}) \in \R^{\ell}$, we have 
\[ |\wh{f}(\theta, \xi)|  \leq \eps \cdot \exp\big((\log n)^{2/3}\big) 
\exp\big(-c(\log n)^{-26}\big( |\theta\eps |^{1/2} + |\xi_1|^{1/2} + \cdots + |\xi_{\ell}|^{1/2} \big) \big), \]
where $c>0$ is an absolute constant.\end{lemma}

As this construction is based on fairly standard techniques, we postpone it to Appendix~\ref{app:smoothbump}.

\subsection{Replacement step}
Here we prove the replacement step described at \eqref{eq:G-replacement} above. For this, we will bolster a weaker version of this replacement step known as the Lindeberg exchange method that gives us the analogue of our result up to probability scales of $n^{-\Omega(1)}$. We include a short proof for completeness.

\begin{theorem}\label{thm:Lindeberg}
Let $X_1,\ldots,X_n$, be independent random variables with mean $0$, variance $1$,  let $Z_1,\ldots ,Z_n$ be iid standard normals, and let $f :\R^n \rightarrow \mathbb{C}$ be a smooth function. Then 
\[ \big| \E\, f(X_1,\ldots,X_n) - \E f(Z_1,\ldots,Z_n)\big| \leq    \max_{1\leq i \leq n} \big[ \E\,|X_i|^3 + \E\, |Z_i|^3\big] \cdot \sum_{i=1}^n \, \bigg\|\frac{d^3f}{dx_i^3}\bigg\|_{\infty}  . \]
\end{theorem}
\begin{proof}
We bound $|\mb{E}\, f(X_1,\ldots,X_n) - f(Z_1,\ldots,Z_n) |$ above by the increments\[ \sum_{i=1}^{n}\big|\mb{E}\, f(X_1,\ldots,X_{i-1},Z_{i},\ldots,Z_n) - \mb{E}\, f(X_1,\ldots,X_{i},Z_{i+1},\ldots,Z_n) \big| \]
We now condition on the worst case values of the $X_j,Z_j$ for all but the $i$th coordinate
\[ \leq  \sum_{i=1}^{n}\max_{t_j \in \mb{R}}\big|\mb{E}_{X_i,Z_i}\, f(t_1,\ldots,t_{i-1},Z_{i},t_{i+1},\ldots,t_n) - f(t_1,\ldots,t_{i-1},X_{i},t_{i+1},\ldots,t_n) \big|.\]
Let $g_{t}(y) = f(t_1,\ldots,t_{i-1},y,t_{i+1},\ldots,t_n)$ and note, by the remainder form of Taylor's theorem, that
\[\bigg|g_{t}(y) - g_{t}(0) - y\cdot g_{t}'(0) - \frac{y^2}{2}g_{t}''(0)\bigg|\le y^3 \cdot \sup_{z\in \mb{R}}|g_{t}'''(z)|.\]
Taking expectations and summing gives the desired result.
\end{proof}

We now prove our main replacement result. For this, we shall need the following basic estimate. If $Y = (Y_1,\ldots, Y_n)$ is a random vector where $Y_i$ are iid and subgaussian and $u \in \R^n$, we have 
\begin{equation}\label{eq:char-fun-estimate} \E_Y \exp( i\la u , Y \ra ) \leq \exp\big( - c_0\inf_{r \in [1,c_0^{-1}]} \|rv\|^2_{\mb{T}} \big), \end{equation} 
where $c_0>0$ depends only on the random variable $Y_1$. Here $\|v\|^2_{\mb{T}}$ for $v\in \mb{R}^n$ is defined as 
\[\|v\|^2_{\mb{T}} = \sum_{i=1}^{n}\snorm{v_i}^2_{\mb{T}}\]
where $\snorm{v_i}_{\mb{T}} = \min_{z\in \mb{Z}}|v_i-z|$. This precise inequality appears as as \cite[Lemma~5.5]{CJMS24} (applied with $u=0$) and equivalent inequalities date at least to work Rudelson and Vershynin \cite{RV08}.

\begin{lemma}\label{clm:test-function}
Let $M^{\ast}\in \mc{E}^{\ast}$, let $\eps\ge e^{-cn}$ and let $f \in \{f^+,f^-\}$ be as above. Then 
\begin{equation} \label{eq:smooth-replacement} \bigg| \E_X\, f(H)-\mb{E}_K\, f(K)\bigg| \le \eps \cdot n^{-\Omega(1)} .\end{equation}
Here $c>0$ is an absolute constant. 
\end{lemma}
\begin{proof}
Recall the definitions of the random vectors $H, K$ taking values in $\mb{R}^{\ell+1}$,
\[H = \big(\sang{v,X},\sang{v_{n-1},X},\ldots,\sang{v_{n-\ell},X}\big)\text{ and }K = \big(\sang{v,G},\sang{v_{n-1},G},\ldots,\sang{v_{n-\ell},G}\big).\] For convenience, we let $H_0$ and $K_0$ denote the first coordinates of these vectors.

We write $\xi= (\xi_1,\ldots,\xi_{\ell})$, and bound the left hand side of \eqref{eq:smooth-replacement} above by
\begin{equation}\label{eq:smooth-fourier-step} \int_{(\theta,\xi)\in \mb{R}^{\ell+1}} | \wh{f}(\theta,\xi)|\cdot \big| \mb{E}\, e^{-2\pi i( H_0\theta + \la \xi, H\ra)}
 -\mb{E}\, e^{-2\pi i(\theta K_0 + \la \xi,K\ra )} \big| \, d\theta d\xi ,\end{equation}
by using Fourier inversion and the triangle inequality. We now define the rectangular region
\[ \Omega = \big\{ (\theta,\xi) \colon |\theta|\leq \eps^{-1} (\log n)^{28} \text{ and } \|\xi\|_{\infty} \leq (\log n)^{30} \big\},\]
and truncate \eqref{eq:smooth-fourier-step} to $\Omega$. In particular, using the decay condition on $\hat{f}$, we see that \eqref{eq:smooth-fourier-step} is 
\begin{equation}\label{eq:int-to-bound} \lesssim \eps \cdot e^{2(\log n)^{2/3}} \cdot \int_{(\theta,\xi) \in \Omega}\big| \mb{E}\, e^{-2\pi i( H_0\theta + \la \xi, H\ra)}
 -\mb{E}\, e^{-2\pi i(\theta K_0 + \la \xi,K\ra )}\big| \, d\theta d\xi + \eps n^{-\omega(1)}.  \end{equation}  We now write
\[ \big| \mb{E}_X\, e^{-2\pi i( H_0\theta + \la \xi, H\ra)}
 -\mb{E}_G\, e^{-2\pi i(\theta K_0 + \la \xi,K\ra )}\big|  = \big|\psi(\theta, \xi) - \wt{\psi}(\theta, \xi) \big| \] and address the integral in \eqref{eq:int-to-bound} in two different ranges.

In the case $|\theta|\le (\log n)^{30}$, we look to apply Theorem~\ref{thm:Lindeberg}, for each fixed $(\theta,\xi)$ to bound the integrand. Write $x = (x_1,\ldots,x_n)$ and $u = u(\theta,\xi) = \theta v + \sum_i \xi_i v_i$ and define 
\[ g(x) = g_{\theta,\xi}(x) = \exp\bigg( -2\pi i \bigg\la x , \theta v + \sum_{i=1}^{\ell} \xi_i v_i \bigg\ra \bigg) = \exp\big( -2\pi i \la x, u \ra \big). \]
so that 
\begin{equation}\label{eq:invar}  \big|\psi(\theta, \xi) - \wt{\psi}(\theta, \xi) \big| = \big|\E_X\, g(X) - \E_G\, g(G)\big|.  \end{equation} 
Now note that 
\[ \|u\|_2^2 = \bigg\la \theta v + \sum_{i=1}^{\ell} \xi_i v_i , \theta v + \sum_{i=1}^{\ell} \xi_i v_i \bigg\ra = \|\theta\|_2^2 + \sum_{i=1}^{\ell} |\xi_i|_2^2 \leq (\log n)^{60}\ell. \] Since $|\theta| \leq (\log n)^{30} \text{ and } M^{\ast} \in \cE $ we have that 
\[ \| u \|_{\infty} =  \bigg\| \theta v + \sum_{i=1}^{\ell}\xi_i v_i \bigg \|_{\infty} \leq |\theta| \|v\|_{\infty} + \sum_{i=1}^{\ell} |\xi_i| \|v_i\|_{\infty} \le n^{-\Omega(1)}.\] 
So we apply Theorem~\ref{thm:Lindeberg} to see that  
\[ \big|\E_X\, g(X) - \E_G\, g(G)\big| \lesssim \sum_{i=1}^n \bigg\|\frac{d^3g}{dx_i^3}\bigg\| \leq \sum_{i=1}^{n} u_i^3 \leq  \|u\|_{\infty} \| u \|_2 ^2\leq n^{-\Omega(1)}.  \] Thus
\[ \int_{(\theta,\xi) \in \Omega,~|\theta|\leq (\log n)^{30}}   \big|\psi(\theta, \xi) - \wt{\psi}(\theta, \xi) \big| \, d\theta d\xi \leq (\log n)^{28\ell + 30} n^{-\Omega(1)} = n^{-\Omega(1)+o(1)}. \]

We now deal with the case $|\theta|\ge (\log n)^{30}$, write $\Omega^+ = \Omega \cap \{ |\theta|\geq (\log n)^{30} \}$, and show that the contributions from both $\psi,\wt{\psi}$ are small. First note  
\[ \int_{(\theta,\xi) \in \Omega^+}|\wt{\psi}(\theta,\xi)| d\theta d\xi= \int_{(\theta,\xi) \in \Omega^+}| \mb{E}\, e^{-2\pi i ( K_0\theta +  \la \xi , K \ra  ) } | d\theta d\xi \le \int_{(\theta,\xi) \in \Omega^+} \exp(-\Omega(\snorm{\xi}_2^2 + |\theta|^2 )\big) d\theta d\xi\leq n^{-\omega(1)}, \]
which is small enough for us.

To bound the contribution from $\psi$, we use that 
\begin{equation}\label{eq:psi-tail} I = \int_{(\theta,\xi) \in \Omega^+ }|\psi(\theta,\xi)|\, d\theta d\xi \leq 
\int_{(\theta,\xi) \in \Omega^{+}}\exp\left( - \inf_{c'\in [1,c_0]}c_0\bnorm{c'(\theta v + \sum_{i=1}^{\ell}\xi_i v_i)}_{\mb{T}}^2\right) \, d\theta d\xi . \end{equation}
This is via the estimate in \eqref{eq:char-fun-estimate}. Now note that since $|c'\theta|\lesssim \eps^{-1}\cdot (\log n)^{30} < e^{cn}$ and thus $\|c'\theta v\|_{\mb{T}} \gtrsim \min\{ |\theta|, \sqrt{n} \}$, by property $\mc{E}_{\mr{lcd}}$. Therefore
\[ \bnorm{c'(\theta v + \sum_{i=1}^{\ell}\xi_i v_i)}_{\mb{T}} \ge \snorm{c'(\theta v)}_{\mb{T}} - \bnorm{c'(\sum_{i=1}^{\ell}\xi_i v_i)}_2
\ge \snorm{\theta v}_{\mb{T}}- (\log n)^{28}\cdot \ell
\gtrsim \min\{ |\theta|, \sqrt{n} \} \]
and therefore the integral on the left hand side of \eqref{eq:psi-tail} is $I \leq n^{-\omega(1)}$, which again is small. Note here the final inequality uses that $|\theta|\ge (\log n)^{30}$ while $\ell = \sqrt{\log n}$. \end{proof}

\subsection{Proof of Lemma~\ref{lem:G-replacement}}

We now finish the proof of Lemma~\ref{lem:G-replacement}.

\begin{proof}[Proof of Lemma~\ref{lem:G-replacement}] Apply the properties guaranteed by Lemma~\ref{lem:bump-function} and Lemma~\ref{clm:test-function} to get
\[ \P_X( \cF(X) ) \leq \mb{E}_X \, f^+(H) + O(\delta_n\eps) 
= \mb{E}_G \, f^+(K) + O(\delta_n\eps) = \mb{E} \, f^{-}(K) + O(\delta_n\eps) \leq \P( \cG(G) ) + O(\delta_n\eps), \]
thus proving \eqref{eq:G-replacement-ub}. The proof of \eqref{eq:G-replacement-lb} is similar. \end{proof}

\section{Return to the singular value problem and proof of Theorem~\ref{thm:main}} \label{sec:return-to-sing-value}

In this section we prove our main theorem, Theorem~\ref{thm:main}. We do this by first proving the following lemma, which, relates our main quantity of interest to the probability that the least singular value of an $n\times n$ \emph{Gaussian} matrix is small. Together with our work from earlier sections and the work of Edelman \cite{Ede88}, this will allow us to complete the proof of our main theorem. We recall that $\chi(X)$ is defined as a function of a random vector $X$ in \eqref{eq:def-delta-ell-chi}, $v$ is a kernel vector of $M^{\ast}$, $\ell = \sqrt{\log N}$, and that $Z$ is a vector of standard Gaussians.

\begin{lemma}\label{lem:back-to-sing-value}
For $0\leq \eps < n^{-c}$, we have 
 \begin{equation}\label{eq:back-to-sing-value} \E_{M^{\ast}}\, \PP_Z\Big( |\sang{v,Z}| \leq \eps n^{-1/2} \cdot \chi(Z) \Big) = \big(1 + O(\delta_n) \big)  \PP_{G}\big( \sigma_n(G) \leq \eps n^{-1/2} \big) + O(\delta_n \eps),  \end{equation} 
 where $c>0$ is an absolute constant. 
\end{lemma}

To prove this we use the symmetry and smoothness of the distribution of $Z$ to show that we can ``replace'' the small $\eps$ on the left hand side of \eqref{eq:back-to-sing-value} with a much larger (and therefore more tractable) $\eps_0 = n^{-c}$. This step we described at \eqref{eq:scaling-up} in the proof sketch.

\begin{lemma}\label{lem:scaling-up} Let $0 \leq \eps \leq n^{-c}$, let $\eps_0 \in [n^{-c}, n^{-c/2}]$ and let $\sigma_{n-1}(M^{\ast})\geq (\log n)^{-1}n^{-1/2}$. Then 
\[\PP_Z\big(  |\sang{v,Z}| \leq \eps n^{-1/2} \cdot \chi(Z) \big) = (\eps /\eps_0) \cdot \PP_Z\big(  |\sang{v,Z}| \leq \eps_0 n^{-1/2} \chi(Z) \big)  + O(\delta_n\eps). \]
Here $c>0$ is an absolute constant. 
\end{lemma}
\begin{proof}
Since $Z \sim N(0,1)^{\oplus n}$ and the $v_i$ are orthonormal, we can express 
$|\sang{v,Z}| \leq \eps n^{-1/2} \cdot \chi(Z)$ as
\[ |W_0|\leq \eps n^{-1/2} \bigg( \sum_{i=1}^{\ell} \frac{W_i^2}{\sigma_{n-i}(M^{\ast})^2} \bigg)^{\frac{1}{2}}  =: \eps \cdot f(W_1,\ldots W_{\ell}), \] where $(W_0,W_1,\ldots,W_{\ell}) \sim N(0,1)^{\oplus (\ell+1)}$, and thus consider $\PP( W_0 \leq \eps f(W_1,\ldots,W_{\ell}))$.

Using the trivial bound and the assumption $\sigma_{n-1}(M^{\ast}) \geq (\log n)^{-1}n^{-1/2}$, we have
\[ f(W_1,\ldots, W_{\ell}) \leq n^{-1/2} \sigma^{-1}(M^{\ast})^{-1}\ell^{1/2}  \max_{1\leq i\leq \ell} |W_i|  \leq  (\log n )^{5/4}\ \max_{1 \leq i\leq \ell} |W_i| , \] so union bounding over each $W_i$ and using the fact that each $W_i$ is a standard gaussian we have
\begin{equation}\label{eq:tail-bound} \mb{P}(f(W_1,\ldots W_{\ell})\ge t) \le \ell \cdot \mb{P}(|W_1|\ge t/(\log n)^{5/4})\le \log n \cdot e^{-\Omega(t^2/(\log n)^{3})}. \end{equation}

Let $j_0$ be the integer such that $2^{j_0}\le n^{c/4}<2^{j_0+1}$ and abbreviate $f = f(W_1,\ldots, W_{\ell})$ to reduce clutter. We dyadically split the probability  
\begin{equation}\label{eq:dyadic-split}\mb{P}\big(W_0\le \eps \cdot f \big) = \mb{P}\big(W_0\le \eps \cdot f\wedge f <2^{j_0}\big) + \sum_{j\ge j_0}\mb{P}\big(W_0\le \eps \cdot f \wedge f\in [2^{j},2^{j+1})\big)\end{equation} and deal with the first and remainder terms separately.

For the first term, note that $\eps  n^{c/4}\le \eps_0 n^{c/4}\le n^{-c/4}$ and that $W_0$ is independent of $f(W_1,\ldots W_{\ell})$, and the Gaussian density function satisfies $(2\pi)^{-1/2} e^{-t^2/2}= (2\pi)^{-1/2}\cdot (1+O(t^2))$ for $t = o(1)$. Therefore we have 
\[ \mb{P}\big(W_0\le \eps \cdot f \wedge f <n^{c/4}\big) 
=\eps/\eps_0\cdot (1\pm n^{-c/8}) \cdot \mb{P}\big(W_0\le \eps_0 \cdot f \wedge f<n^{c/4}\big) \]
We now use that $\mb{P}(f >n^{c/4}) = n^{-\omega(1)}$ by \eqref{eq:tail-bound} and $(\eps/\eps_0 ) n^{-\omega(1)} = O(\delta_n \eps)$ to see that 
\[ \mb{P}\big(W_0\le \eps_0 \cdot f \wedge f<n^{c/4}\big) = \eps/\eps_0\cdot (1\pm \delta_n) \cdot \mb{P}(W_0\le \eps_0 \cdot f) + O(\delta_n \eps).\]

Thus it only remains to show the sum in \eqref{eq:dyadic-split} is $O( \delta_n \eps)$. For this, write 
\[ \sum_{j\ge j_0}\mb{P}(W_0\le \eps \cdot f \wedge f\in [2^{j},2^{j+1}]) \le \sum_{j\ge j_0}\mb{P}(W_0\le \eps \cdot 2^{j+1}) \cdot \mb{P}( f(W_1,\ldots W_{\ell})\ge 2^{j})\]
and then use that $\mb{P}[W_0\le t]\lesssim t$ and \eqref{eq:tail-bound} to see that the above is
\[ \le \sum_{j\ge j_0}\eps \cdot 2^{j+1} \cdot e^{-\Omega(4^{j}/(\log n)^{3/2})} \ll \sum_{j\ge j_0}\eps \cdot 2^{-j} = O(\delta_n\eps),\]
as desired.

\end{proof}

We now properly state the universality theorem of Tao and Vu, for the least singular value. The following is Theorem 6.7 in \cite{TV10}.

\begin{theorem}\label{thm:Tao-Vu-LSV-universal} Let $\wt{M}$ be an $n \times n$ random matrix, where $\wt{M}_{ij}$ are independent with mean $0$, variance $1$ and subgaussian constants $\|\wt{M}_{ij}\|_{\psi_2} \leq B$. Then
\[ \PP\big( \sigma_{n}(\wt{M}) \leq \eps n^{-1/2} \big)  = \PP( \sigma_{n}(G) \leq \eps n^{-1/2} \big) + O(n^{-c}), \] for some $c>0$, depending only on $B$.
\end{theorem}

We also need the following theorem of Edelman \cite{Ede88}, mentioned in the introduction.

\begin{theorem}\label{thm:edelman} Let $G$ be an $n\times n$ random matrix with iid entries distributed as $N(0,1)$. Then 
\[ \PP\big(  \sigma_n(G) \leq \eps n^{-1/2} \big) \leq \eps .\] 
\end{theorem}

We are now in a position to prove Lemma~\ref{lem:back-to-sing-value}.

\begin{proof}[Proof of Lemma~\ref{lem:back-to-sing-value}] 
Write $\cF_{\eps} = \{ |\sang{v,Z}| \leq \eps n^{-1/2} \cdot \chi(Z)  \}$. Using Lemma~\ref{lem:scaling-up} with the right hand side of \eqref{eq:back-to-sing-value} with $\eps_0 = n^{-c}$, we have
\[ \E_{M^{\ast}} \PP_Z\big( \cF_{\eps}  \big) = (\eps /\eps_0) \cdot \E_{M^{\ast}} \1\big(  M^{\ast} \in \cE^{\ast} \big) \cdot \PP_Z\big( \cF_{\eps_0}  \big) + O(\delta_n\eps).  \] Here we have used Lemma~\ref{lem:E-star-negligible} to deal with the case that $\s_{n-1}(M^{\ast}) < (\log n)^{-1}n^{-1/2}$. Now apply
Lemma~\ref{lem:main-geometric-lemma} to see that 
\[\PP_{\wt{M}} \big( \s_n(\wt{M}) \leq \big(1-\delta_n\big) \eps_0 n^{-1/2}\big) + O(\delta_n\eps) \leq  \E_{M^{\ast}} \PP_Z \big( \cF_{\eps_0} \big)  \leq \PP_{\wt{M}}\big( \s_n(\wt{M}) \leq \big(1+\delta_n\big) \eps_0 n^{-1/2}\big) + O(\delta_n\eps), \]
where $\wt{M}$ is the random matrix $M^{\ast}$ with a Gaussian row $Z$ added. We now apply Theorem~\ref{thm:Tao-Vu-LSV-universal} to the right hand side of the above to see  
\[\PP_{\wt{M}}\big( \s_n(\wt{M}) \leq \big(1+\delta_n\big) \eps_0 n^{-1/2}\big) =  \PP_G( \sigma_{n}(G) \leq (1 + \delta_n) \eps_0 n^{-1/2} \big) + O(n^{-c}) = \big(1 + O(\delta_n) \big)\eps_0,  \]
and similarly for the left hand side. Now apply Theorem~\ref{thm:edelman} to obtain the desired result. \end{proof}

\subsection{Proof of Theorem~\ref{thm:main} and Theorem~\ref{cor:ST}}

Here we prove our main theorem, Theorem~\ref{thm:main}, by snapping together the results developed in the previous sections. 

\begin{proof}[Proof of Theorem~\ref{thm:main} ] We prove the ``$\leq$'' direction of our main theorem and note that the other direction is similar. We may assume that 
$e^{-cn} < \eps < n^{-c}$, otherwise we can apply the results of Rudelson and Vershynin or Tao and Vu, respectively. We first apply Lemma~\ref{lem:main-geometric-lemma} and write
\[ \PP\big( \s_n(M) \leq \eps n^{-1/2}  \big) \leq  \E_{M^\ast} \PP_X\big( |\la v, X \ra| \leq \big( 1+\delta_n \big) \eps n^{-1/2} \chi(X) \big) + O(\delta_n\eps). \]
Now apply our Gaussian replacement step, Lemma~\ref{lem:G-replacement}, to write
\[ \E_{M^{\ast}} \PP_X\big( |\sang{v,X}| \leq \eps n^{-1/2} \chi(X) \big) \leq \E_{M^{\ast}} \PP_Z\big( |\sang{v,Z}| \leq \big(1+ \delta_n \big)\eps n^{-1/2} \cdot \chi(Z) \big) + O(\delta_n\eps), \]
where we used Lemma~\ref{lem:E-star-negligible}, to deal with the event $M^{\ast} \not\in \cE^{\ast}$. Now we apply Lemma~\ref{lem:back-to-sing-value} to write
\[\E_{M^{\ast}}\, \PP_Z\Big( |\sang{v,Z}| \leq \big(1+\delta_n\big) \eps n^{-1/2} \cdot \chi(Z) \Big) = \big(1 + O(\delta_n) \big)  \PP_{G}\big( \sigma_n(G) \leq  \big(1+\delta_n\big)\eps n^{-1/2} \big) + O(\delta_n \eps), \]
thus completing the proof.
\end{proof}

We now easily derive our approximate Spielman--Teng conjecture. 

\begin{proof}[Proof of Theorem~\ref{cor:ST}]
Apply Theorem~\ref{thm:main} and Theorem~\ref{thm:edelman} and write
\[ \PP\big( \s_n(M) \leq \eps n^{-1/2}  \big) \leq \big(1+o(\delta_n)\big)\PP\big( \s_n(G) \leq \eps n^{-1/2}  \big)  + e^{-\Omega(n)} \leq \big(1+\delta_n \big)\eps + e^{-\Omega(n)} ,\]
where $G$ is a random $n\times n$ matrix with iid standard normal entries. 
\end{proof}

\section*{Acknowledgements}

A portion of this work was completed when Sawhney was visiting Cambridge University under the Churchill Scholarship. Sah and Sawhney were supported by NSF Graduate Research Fellowship Program DGE-2141064.

We would like to thank Vishesh Jain, Marcus Michelen, and Robert Morris for comments on the presentation of this paper. Sawhney would also like to thank the combinatorics group at the University of Cambridge for providing a stimulating working environment during his stay. Finally we would like to thank the anonymous referees for useful comments and corrections on the manuscript. 

\appendix 

\section{The flatness of the \texorpdfstring{$v_i$}{}}\label{app:flat-v_i}

We first require which follows from interlacing and the proof of \cite[Lemma~4.1]{TV10}.
\begin{lemma}\label{lem:many-small}
There exist $c \in (0,1/4)$ and $c' > 0$ (potentially depending on the subgaussian constant of $\xi$) such that with probability $1-\exp(-c'n^{c})$, there are at least $c'n^{1-c}$ singular values of $M_n$ in $[n^{1/2-c}/2,n^{1/2-c}]$.
\end{lemma}

We now prove Lemma~\ref{lem:unstructured-vector}.
\begin{proof}[{Proof of Lemma~\ref{lem:unstructured-vector}}]
Let $c$ be as in \cref{lem:many-small}; and suppose that there exists $v\in \mb{S}^{n}$ with $\snorm{v}_{\infty}\ge n^{-c/4}$ and $\snorm{M_n^{\ast}v}_2\le 1$. Applying the union bound by symmetry we may assume that $|v_n|\ge n^{-c/4}$ (at the cost of paying a negligible factor $n$ in the probability). Let $\wt{M}$ denote the first $n-1$ columns of $M_n^{\ast}$, $\wt{X}$ denote the last column of $M_n^{\ast}$, and $v' = (v_1,\ldots,v_{n-1})$. Since $\snorm{\wt{M}}_{\mr{op}}\le 5\sqrt{n}$ with exponentially good probability by \cite[Theorem~C.1]{CJMS22}, if $\snorm{M_n^\ast v}_2\le 1$ then
\[\snorm{\wt{M}^T(\wt{M}v' + \wt{X}v_n)}_2\le 5\sqrt{n}.\]
Let $F$ denote the set of singular values of size between $[2^{-2}n^{1/2-c},2n^{1/2-c}]$. By \cref{lem:many-small} applied to $\wt{M}$, with very high probability, we have $|F|\ge c'n^{1-c}$. Letting $\pi_F$ denote the projection onto the corresponding left-singular vectors of $\wt{M}$ and note that 
\[\snorm{\pi_F\wt{M}^T(\wt{M}v' + \wt{X}v_n)}_2 = \snorm{\pi_F\wt{M}^T\wt{M}v' + \pi_F\wt{M}^T\wt{X}v_n}_2\le 5\sqrt{n}.\]
As $v_n\ge n^{-c/4}$, we have
\[\snorm{\pi_F\wt{M}^T\wt{X}}_2\le v_n^{-1}(\snorm{\pi_F\wt{M}^T\wt{M}v'}_2 + 5\sqrt{n})\le 5n^{1-7c/4}.\]
Furthermore note that
\[\snorm{\pi_F\wt{M}^T}_{\on{HS}}^2 =\sum_{j\in F}\sigma_k(\wt{M})^2\gtrsim c'n^{2-3c}.\]
The desired result then follows via the Hanson--Wright inequality. In particular, \cite[Theorem~2.1]{RV13} derives that for a subGaussian vector $X$ that $\snorm{AX}_2-\snorm{A}_{\on{HS}}$ is subGaussian with parameter $O(\snorm{A}_{\on{op}})$. In our application, $A=\pi_F\wt{M}^T$, $\snorm{\pi_F\wt{M}^T}_{\on{HS}}\gtrsim n^{1-3c/2}$, and $\snorm{\pi_F\wt{M}^T}_{\on{op}}\le \snorm{\wt{M}}_{\on{op}}\le 5\sqrt{n}$.
\end{proof}

\section{Proof of our negative dependence lemma}\label{app:negative-dependence}

\begin{theorem} Let $u \in \R^n$ satisfy $\|u\|_2 =1$ and let $w_1,\ldots, w_k \in \R^n$ be orthogonal unit vectors. Furthermore let $Y = (Y_1,\ldots,Y_n)$
be a random vector where the $Y_i$ are iid, mean $0$ and subgaussian. 

If $\on{LCD}(u) \geq C/\eps$ and $k\cdot \snorm{u}_{\infty} + \sum_{i=1}^{k}\snorm{w_i}_{\infty}\le (\log n)^{-3}$ then 
\[\mb{P}\bigg( |\sang{u,Y}|\le \eps\, \wedge\,  \sum_{j=1}^{k}|\sang{w_i,Y}|^2\le ck\bigg)\le C(\eps e^{-ck} + e^{-c n}),\]
where $C,c>0$ depend only on the subgaussian constant $\|Y_i\|_{\psi_2}$. \end{theorem}

\begin{proof}
We may assume that $\eps\le e^{-ck/2}$; else adjusting $c$ the result follows immediately from the Hanson--Wright inequality (see \cite[Corollary~3.1]{RV13}). Furthermore we may assume that $\eps\ge e^{-c\alpha n/2}$ by adjusting $c$ appropriately. Furthermore note that we may assume that $n$ is a sufficiently large absolute constant else the estimate is vacous. For the remainder of the argument, let $c$ be a sufficiently small constant to be chosen later. 
We first bound 
\begin{equation}\label{eq:pf-decouple-0} \mb{P}\bigg(|\sang{u,Y}|\le \eps \wedge \sum_{j=1}^{k}|\sang{w_i,Y}|^2\le ck\bigg) \le 2^{k}\max_{|S| = k/2}\mb{P}\big(|\sang{u,Y}|\le \eps \wedge \bigwedge_{j\in S}|\sang{w_i,Y}|\le 2\sqrt{c}\big) . \end{equation}
We now approximate the event on the right hand side using the smooth bump function 
\[ f(t) = (\sin(2\pi t)/(\pi t))^2 \quad  \text{ for which } \quad \wh{f}(\theta) = (\mbm{1}_{[-1,1]}\ast \mbm{1}_{[-1,1]})(\theta).\]
Assume, without loss, that $S = \{1,\ldots, k/2\}$.
Note that $f(t)\ge 1$ for $t\le 1/4$. Thus we have 
\begin{equation}\label{eq:pf-decouple1} \mb{P}\big(|\sang{u,Y}|\le \eps \wedge \bigwedge_{j\in S}|\sang{w_i,Y}|\le 2\sqrt{c}\big) \leq  \mb{E}\, f\big(\sang{u,Y}(4\eps)^{-1}\big)\prod_{j\in S}f\big(\sang{w_i,Y}(64c)^{-1/2}\big). \end{equation} Now note that the characteristic function of the random variable 
$\big(\la v,Y\ra , \la w_1,Y \ra,\ldots ,\la w_{k/2}, Y \ra \big)$ is
\[ \phi(\theta,\xi_1,\ldots,\xi_{k/2})  = \mb{E}\, \exp\bigg(i\theta\sang{Y,v} + \sum_{1\leq j \leq k/2}i \xi_j\sang{w_j,Y}\bigg) . \]
After putting $\xi = (\xi_1,\ldots,\xi_{k/2})$, we may rewrite the expression on the right hand side of \eqref{eq:pf-decouple1} using the Fourier transform as
\begin{equation}\label{eq:pf-decouple2} \int_{\mb{R}^{k/2+1}}4\eps\wh{f}\big(4\eps \theta\big)\prod_{j\in S}\big((64c)^{1/2} \cdot \wh{f}\big(8\sqrt{c}\xi_j\big)\big) \cdot \phi(\theta,\xi) ~d\theta d\xi \leq \eps 16^{k+1}c^{k/4}\int_{(\theta,\xi)\in B}|\phi(\theta,\xi)|~d\theta d\xi,\end{equation} 
where we define $B = \{ (\theta,\xi)\colon |\theta|\le \eps^{-1}, |\xi_j|\le c^{-1/2}, \forall j \}$ and used that $\wh{f}$ is supported on $[-2,2]$.

We now bound the integral in \eqref{eq:pf-decouple2} in two regimes, based on $|\theta|$. If $|\theta|\le k \cdot \log n$ we define  
\[u(\theta) = \theta u + \sum_{j\leq k/2}\xi_j w_j \qquad \text{ and note } \qquad \| u(\theta)\|_{\infty}\le (\log n)^{-1},\]
by using the bounds on $k$, $\|u\|_{\infty}$ and $\|w_i\|_{\infty}$. 
Thus, $\|r \cdot u(\theta,\xi)\|_{\mb{T}} = \|r \cdot u(\theta,\xi)\|_{2}$ for all scalars $r < (\log n) /2$. Thus, using the inequality in \eqref{eq:char-fun-estimate}, we have
\[ |\phi(\theta,\xi)| = \big| \E_Y \exp\big( i \la u(\theta,\xi) , Y \ra \big)\big| \leq \exp\big( -c_0\inf_{r \in [1,c_0^{-1}]}\|r u(\theta,\xi)\|^2_{\mb{T}} \big)  \leq  \exp\big(-c_0|\theta|^2 -c_0\|\xi\|^2_{2} \big) ,   \]
where $c_0>0$ is a constant depending only on the distribution of $Y_1$ and were we have used that $u,,w_1,\ldots,w_k$ are orthonormal for the last inequality. As a result, if we set $B^- = \{ (\theta,\xi) \in B\colon |\theta| \leq k\log n\}$, we have 
\begin{equation}\label{eq:pf-decouple-smalltheta} \int_{B^{-}} |\phi(\theta,\xi)| \, d\theta d\xi = \exp(O(k)),\end{equation}
where the implicit constant depends only on the distribution of $Y_i$.

We now consider the other range of $\theta$, where $k\log n \leq |\theta| \leq \eps^{-1}$. We have that
\begin{align*}
\bnorm{u(\theta,\xi)}_{\mb{T}} &\ge \snorm{\theta u }_{\mb{T}} - \bnorm{\sum_{j}\xi_j w_j}_{2}\\
&\ge \min\{\gamma\snorm{\theta u}_2,\sqrt{\alpha n}\} - k^{1/2}c^{-1/2}
\ge \min\{\gamma\snorm{\theta u}_2,\sqrt{\alpha n}\}/2 = \min\{\gamma\Theta,\sqrt{\alpha n}\}/2,
\end{align*}
where the second inequality holds by the condition on the last common denomonator of $u$ and the $L_{\infty}$ bound on the $w_i$. The last inequality holds when $n$ is sufficiently large with respect to $\gamma$. 

We now use this information along with the inequality in \eqref{eq:char-fun-estimate} to write 
\[ |\phi(\theta, \xi)| \leq \exp\big(-c_0\inf_{r \in [1,c_0^{-1}]}\|r \cdot u(\theta,\xi)\|^2_{\mb{T}}\big) \leq \exp(-\Omega(\min\{\gamma \snorm{\Theta v}_2^2,\alpha n)) . \] 
Integrating this expression over $B^+ = \{ (\theta,\xi)\colon k\log n \leq |\theta| \leq \eps^{-1} \}$ gives
\begin{align} \int_{B^+} |\phi(\theta,\xi)|\, d\theta d\xi &\leq (1/c)^{O(k)}\int_{|\Theta|\ge k\log n}e^{-\Omega(\gamma^2 \cdot \Theta^2)}~d\Theta +|B^+| \cdot e^{-\Omega(n)} \\
&\leq (1/c)^{O(k)} \cdot e^{-\Omega(k^2(\log n)^2)} + \eps^{-1}(1/c)^{O(k)} e^{-\Omega(n)} = o(1). \label{eq:pf-decouple-largetheta}
\end{align}
Here we have used that $\eps > e^{-c\alpha n}$.

We now put \eqref{eq:pf-decouple-0} together with \eqref{eq:pf-decouple2} and our estimates on the integral \eqref{eq:pf-decouple-smalltheta} and \eqref{eq:pf-decouple-largetheta} to get
\[ \mb{P}\bigg(|\sang{u,Y}|\le \eps \wedge \sum_{j=1}^{k}|\sang{w_i,Y}|^2\le ck\bigg) \leq \eps c^{k/4} e^{\Omega(k)}.\]
Thus choosing $c>0$ to be sufficiently small completes the proof.  
\end{proof}

\section{Construction of smooth bump functions}\label{app:smoothbump}

We will need the construction of a bump function with sufficiently fast Fourier decay. 
\begin{lemma}\label{lem:analysis}
There exists a function $\psi\colon\mb{R}\to\mb{R}$ such that $\psi\ge 0$, $\hat{\psi}\ge 0$,  $\on{supp}(\hat{\psi})\in [-1,1]$, $\int_\mb{R} \psi(x)~dx= 1$, and $\psi(x)\le \exp(-c(|x|+1)^{1/2})$, where $c>0$ is an absolute constant. 
\end{lemma}
\begin{proof}
Omitting the first and fourth bullet point, $\wh{\psi^{\ast}}(x) = \exp(-(1-4x^2)^{-1})\mbm{1}_{|x|\le 1/2}$ is sufficient. The necessary Fourier decay property is proved in \cite{Joh15}. Taking $\wh{\psi}(x) = (\wh{\psi^{\ast}}\ast \wh{\psi^{\ast}})(x)$ and normalizing fixes the remaining conditions.
\end{proof}

We now work towards defining the bump functions that are required  Lemma~\ref{lem:bump-function}. To construct these functions, let $(x, y_1,\ldots,y_{\ell}) \in \R^{\ell + 1}$; and let $y = (y_1,\ldots,y_{\ell})$. Define the functions
\[\chi(y) = \chi_{M^{\ast}}= \left( \sum_{i=1}^{\ell}\frac{y_i^2}{\s_{n-i}(M^{\ast})^2}\right)^{1/2} \quad \text{ and } \quad B_r(x,y) = \mbm{1}\bigg((1 + r\delta_n)\cdot \eps n^{-1/2} \cdot \chi(y) \ge x \bigg).\]
Here we will only use $|r| \leq 1$. Note that the probabilities in Lemma~\ref{lem:G-replacement} can be expressed as 
\[\E_H B_{1}(H) =  \PP_X\big( |\sang{v,X}| \leq (1+\delta_n) \eps n^{-1/2} \cdot \chi(X) \big);  \E_K B_{-1}(K)  =  \PP_X\big( |\sang{v,X}| \leq (1-\delta_n)\eps n^{-1/2} \cdot \chi(X) \big).  \]

We now truncate the indicators $B_r(x,y)$ in two ways. Define $\tau = (\log n)^{7}$, $\tau' = (\log n)^{25}$ and 
\[
Q(x,y) = \mbm{1}\big( \snorm{y}_{\infty}\le \log n \wedge |x|\le \eps \cdot (\log n)^5\big) \quad \text{ and } \quad 
C(x,y) = \mbm{1}\bigg( \sum_{i\leq  (\log\log n)^2 } y_i^2\ge 1\bigg).\]
We now define the ``smoothed rectangle''
\[ \rho(x,y_1,\ldots,y_{\ell}) = (\eps^{-1} \cdot \tau')\cdot \psi(x \tau'\eps^{-1})  \prod_{i=1}^{\ell} \tau'\cdot \psi(y_i \cdot \tau') \] where $\psi$ is as in Lemma~\ref{lem:analysis}. We then define our functions $f^+,f^-$ as 
\[
f^{+}(y) = (B_{r + \tau^{-1}}(x,y)Q(x,y)C(x,y))\ast \rho \quad \text{ and } \quad 
f^{-}(y) = (B_{r - \tau^{-1}}(x,y)Q(x,y)C(x,y)) \ast \rho. \]
We now prove Lemma~\ref{lem:bump-function} by checking that the functions $f^+,f^-$ satisfy the desired properties. It is easiest to check the final property in Lemma~\ref{lem:bump-function} first. 

\begin{claim}
For $f \in \{f^-,f^+\}$ and all $\theta \in \R$, $\xi = (\xi_1,\ldots,\xi_{\ell}) \in \R^{\ell}$, we have 
\[ |\wh{f}(\theta, \xi)|  \leq \eps \cdot \exp\big((\log n)^{2/3}\big) 
\exp\big(-c(\log n)^{-26}\big( |\theta/\eps |^{1/2} + |\xi_1|^{1/2} + \cdots + |\xi_{\ell}|^{1/2} \big) \big), \]
where $c>0$ is an absolute constant.
\end{claim}
\begin{proof} We check the claim for $f = f^+$ and note the case $f = f^-$ is similar. We have that 
\[ |\wh{f^{+}}(\theta,\xi)| 
\le \snorm{B_{r + \tau^{-1}}(x,y)Q(x,y)C(x,y)}_{L^1(\mb{R}^{\ell + 1})} \cdot |\wh{\rho}(\theta,\xi)| \leq \snorm{Q(x,y)}_{L^1(\mb{R}^{\ell + 1})} \cdot |\wh{\rho}(\theta,\xi)|,\]
where we used $\wh{f \ast g} = \wh{f} \cdot \wh {g}$ and that $\|\hat{f}\|_{\infty} \leq \|f\|_1$. Recalling the definition of $Q$, the above is 
\[\le \eps \cdot (\log n)^{\ell + 10} \cdot |\hat{\rho}(\theta,\xi)|\leq \eps \cdot \exp\big((\log n)^{2/3}\big) 
\exp\big(-c(\log n)^{-26}\big( |\theta/\eps |^{1/2} + |\xi_1|^{1/2} + \cdots + |\xi_{\ell}|^{1/2} \big) \big), \]
where we have used the definition of $\rho$ for the last inequality.\end{proof}

For the next two claims we use the following point-wise inequalities. 
\begin{equation}\label{eq:f+ptwise} B_{r}(x,y)Q(2x, 2y)C(x/2,y/2)\le f^{+}(y) \le B_{r + 2\tau^{-1}}(x,y)Q(x/2,y/2)C(2x,2y)
+E(x,y)\end{equation}
and
\begin{equation}\label{eq:f-ptwise}  (1-n^{-\omega(1)})B_{r - 2\tau^{-1}}(x,y)Q(2x,2y)C(x/2,y/2) \le f^{-}(y) \le B_{r}(x,y)Q(x/2,y/2)C(2x,2y)
+E(x,y).\end{equation}
where $E(x,y)$ is the ``error'' function 
\[E(x,y) =\sum_{t\ge (\log n)^{4}}e^{-t^{1/3}}Q(x/t,y/t).\]

\begin{claim}Let $G \sim N(0,1)^{\otimes n}$. 
We have $| \E_G\, f^{+}(G) -  f^{-}(G)  | = O(\delta_n\eps)$.
\end{claim}
\begin{proof} Recall that $K = ( \la v,X \ra, \la v_{n-1},X\ra, \ldots , \la v_{n-\ell},X\ra )$ which is distributed as $N(0,1)^{\otimes (\ell+1)}$, since the $v_i$ are orthonormal. Now note that 
\begin{equation}\label{eq:Qbound} \mb{E}\, Q(K) \lesssim \eps \cdot (\log n)^{4} \quad \text{ and similarly }  \quad \mb{E}\, E(K) \le \eps \cdot n^{-\omega(1)} ,\end{equation}
via standard Gaussian estimates. So from \eqref{eq:f+ptwise}, \eqref{eq:f-ptwise} and \eqref{eq:Qbound} it is enough to show 
\[\mb{E}\, \big|B_{r + 2\tau^{-1}}(K)Q(K/2)C(2K) - B_{r - 2\tau^{-1}}(K)Q(2K)C(K/2)\big| = O(\delta_n \eps).\]
Towards this goal, we note that the left hand side is at most
\[ \mb{E}\big| Q(K/2) - Q(2K)\big| + \mb{E}\big|Q(2K)(C(2K)-C(K/2))\big| + \mb{E}\big|Q(2K)C(K/2)(B_{r + 2\tau^{-1}}(K)-B_{r -2\tau^{-1}}(K))\big|.\] 
The first two terms are $O(\delta_n\eps)$ by a direct Gaussian computation. So the above is  
\[ \leq 
\P\big(|\sang{v,G} - r\chi(K)| \leq 2\tau^{-1}\delta_n \eps n^{-1/2}\chi(K)\big) + O(\delta_n \eps)
\leq 2\tau^{-1}\delta_n \cdot \eps n^{-1/2} \E_K \chi(K) + O(\delta_n \eps).\]
Now since $M^{\ast}\in \mc{E}^{\ast}$, we have that $\sigma_{n-i} \geq (\log n)^{-3} n^{-1/2}$ and thus 
\[ \E_K \chi(K)\leq \big( \E_K\, \chi(K)^2 \big)^{\frac{1}{2}} \leq n^{1/2}(\log n)^{3}\ell^{1/2}.\] Using this in the above completes the proof of the claim. 
\end{proof}

We now complete the proof of Lemma~\ref{lem:bump-function} with the following claim. 

\begin{claim} For $f \in \{ f^+,f^-\}$, we have  
\begin{equation}\label{eq:clm-P(cF)}  \PP_X( \cF(X) ) \leq \E_X\, f(H) +O(\delta_n\eps) \qquad \emph{ and } \qquad \E_G\, f(K) \leq \PP_G( \cF(G) ) +O(\delta_n\eps). \end{equation}
\end{claim}
\begin{proof} Here we prove the first inequality in \eqref{eq:clm-P(cF)} for $f = f^+$ and note that the other cases are similar or easier. Note it is enough to show that 
\[ \mb{E}[f^{+}(H) - B_r(H)] \geq -\eps (\log n)^{-\omega(1)} \quad \text{ and } \quad  \mb{E}[B_r(H) - f^{-}(H)] \geq -\eps (\log n)^{\omega(1)}.\] 
For the first inequality in the above, we write
\[\mb{E}[f^{+}(H) - B_r(H)]\ge \mb{E}[B_r(H)( Q(2H)C(H/2)-1)], \]
and note that 
\[ \chi(y)\leq (\sigma_{n-1}(M^{\ast}))^{-1} \ell \max_{i\le \ell}|y_i| \leq (\log n)^{11/4} \max_{i\leq \ell} |y_i|, \]
where we used that $M^{\ast} \in \cE^{\ast}$ for the second inequality and thus 
\[
B_r(x,y) \le \mbm{1}\big( \eps (\log n)^{11/4}\max_{1\le i\le \ell}|y_i|\ge x\big).\]
Thus if $B_r(H) = 1$ and $|\sang{v,X}|\ge \eps \cdot (\log n)^{5}/2$ we have that $\max_{1\le i\le \ell}|\sang{H,v}|\ge (\log n)$. Therefore
\begin{align*}
\mb{E}&[f^{+}(H) - B_r(H)]\ge \mb{E}[B_r(H)( Q(2H)C(H/2)-1)]\\
&\ge -\mb{E}\bigg[\eps (\log n)^{11/4}\max_{1\le i\le \ell}|\sang{H,v_i}|\ge |\sang{H,v}| \wedge \bigg(\bigvee_{1\le i\le \ell}|\sang{H,v_i}|\ge (\log n)/2 \vee \sum_{1\le i\le (\log\log n)^2}|\sang{H,v_i}|^2\le 1/2\bigg)\bigg]\\
&\ge -\mb{E}\bigg[\eps (\log n)^{11/4}\max_{1\le i\le \ell}|\sang{H,v_i}|\ge |\sang{H,v}| \wedge \bigvee_{1\le i\le \ell}|\sang{H,v_i}|\ge \log n\bigg]\\
&\qquad\qquad\qquad- \mb{E}\bigg[\eps (\log n)^{15/4}\ge |\sang{H,v}|\vee\sum_{1\le i\le (\log\log n)^2}|\sang{H,v_i}|^2\le 1/2\bigg]\\
&\ge -\eps \cdot (\log n)^{-\omega(1)};
\end{align*}
here we have used Theorem~\ref{thm:decor} and Theorem~\ref{thm:decor-harder}. For the corresponding lower bound, note that 
\begin{align*}
\mb{E}&[B_r(H) - f^{-}(H)]\ge \mb{E}[B_r(H)(1-Q(H/2)C(2H))] -\mb{E}[E(H)]\ge -\mb{E}[E(H)]\\
& = -\sum_{t\ge (\log n)^{4}}e^{-t^{1/3}}\mb{E}[Q(H/t)]\ge -\sum_{t\ge (\log n)^{4}}e^{-t^{1/3}} \mb{E}[\mbm{1}[|\sang{v,H}|\le t \eps (\log n)^{4}]] \ge -\eps \cdot (\log n)^{-\omega(1)}.
\end{align*}

\end{proof}

\bibliography{main}
\bibliographystyle{abbrv}

\end{document}